\documentclass[reqno, 11pt, a4paper]{amsart}

\usepackage{latexsym,ifthen,xspace}
\usepackage{amsmath,amssymb,amsthm}
\usepackage{enumerate, calc}
\usepackage[colorlinks=true, pdfstartview=FitV, linkcolor=blue,
  citecolor=blue, urlcolor=blue]{hyperref}
\usepackage[text={33pc,605pt},centering]{geometry}    %11pt
\usepackage{graphicx}

\newtheorem{theorem}{Theorem}[section]
\newtheorem{lemma}[theorem]{Lemma}
\newtheorem{prop}[theorem]{Proposition}
\newtheorem{assumption}[theorem]{Assumption}
\newtheorem{corro}[theorem]{Corollary}

\theoremstyle{definition}
\newtheorem{definition}[theorem]{Definition}

\theoremstyle{remark}
\newtheorem{remark}[theorem]{Remark}

\numberwithin{equation}{section}

\newcommand{\al}{\ensuremath{\alpha}}
\newcommand{\be}{\ensuremath{\beta}}
\newcommand{\ga}{\ensuremath{\gamma}}
\newcommand{\de}{\ensuremath{\delta}}
\newcommand{\ep}{\ensuremath{\epsilon}}
\newcommand{\ze}{\ensuremath{\zeta}}

\renewcommand{\th}{\ensuremath{\theta}}

\newcommand{\ka}{\ensuremath{\kappa}}

\newcommand{\si}{\ensuremath{\sigma}}

\newcommand{\om}{\ensuremath{\omega}}

\newcommand{\ve}{\ensuremath{\varepsilon}}

\newcommand{\vp}{\ensuremath{\varphi}}

\newcommand{\De}{\ensuremath{\Delta}}

\newcommand{\Si}{\ensuremath{\Sigma}}

\newcommand{\Om}{\ensuremath{\Omega}}
\newcommand{\cD}{\ensuremath{\mathcal D}}
\newcommand{\cE}{\ensuremath{\mathcal E}}
\newcommand{\cF}{\ensuremath{\mathcal F}}

\newcommand{\cH}{\ensuremath{\mathcal H}}
\newcommand{\cI}{\ensuremath{\mathcal I}}

\newcommand{\cL}{\ensuremath{\mathcal L}}

\newcommand{\cN}{\ensuremath{\mathcal N}}

\newcommand{\cT}{\ensuremath{\mathcal T}}

\newcommand{\bbN}{\ensuremath{\mathbb N}}

\newcommand{\bbP}{\ensuremath{\mathbb P}}

\newcommand{\bbR}{\ensuremath{\mathbb R}}

\newcommand{\bbZ}{\ensuremath{\mathbb Z}}
\newcommand{\me}{\ensuremath{\mathrm{e}}}

\newcommand{\md}{\ensuremath{\mathrm{d}}}
\newcommand{\mD}{\ensuremath{\mathrm{D}}}

\newcommand{\scpr}[3]{%
  \ensuremath{%
    \big\langle
      #1, #2
    \big\rangle_{\raisebox{0.1ex}{$\scriptstyle \ell^{\raisebox{.1ex}{$\scriptscriptstyle 2$}} (#3)$}}
  }
}
\newcommand{\norm}[3]{%
  \ensuremath{%
    \left\lVert
      #1
    \right\rVert_{\raisebox{-.0ex}{$\scriptstyle \ell^{\raisebox{.2ex}{$\scriptscriptstyle #2$}} (#3)$}}
  }
}
\newcommand{\Norm}[2]{%
  \ensuremath{%
    \left\lVert
      #1
    \right\rVert_{\raisebox{-.0ex}{$\scriptstyle #2$}}
  }
}

\DeclareMathOperator{\mean}{\mathbb{E}}
\DeclareMathOperator{\Mean}{\mathrm{E}}
\DeclareMathOperator{\prob}{\mathbb{P}}
\DeclareMathOperator{\Prob}{\mathrm{P}}

\DeclareMathOperator{\supp}{\mathrm{supp}}

\DeclareMathOperator{\sign}{\mathrm{sign}}

\newcommand{\av}[1]{\mathop{\mathrm{av}}(#1)}

\newcommand{\ldef}{\ensuremath{\mathrel{\mathop:}=}}

\newcommand{\indicator}{%
  \ensuremath{%
    \mathchoice{1\mskip-4mu\mathrm l}
    {1\mskip-4mu\mathrm l}
    {1\mskip-4.5mu\mathrm l}
    {1\mskip-5mu\mathrm l}
  }
}

\begin{document}

\title[Quenched CLT for the dynamic RCM]{Quenched invariance principle for random walks with time-dependent ergodic degenerate weights}

%    Remove any unused author tags.

%   author one information
\author{Sebastian Andres}
\address{University of Cambridge}
\curraddr{Wilberforce Road, Cambridge CB3 0WB}
\email{s.andres@statslab.cam.ac.uk}
\thanks{}

%    author two information
\author{Alberto Chiarini}
\address{Aix-Marseille Universit\'e (I2M)}
\curraddr{32, rue Joliot Curie, Marseille}
\email{alberto.chiarini@univ-amu.fr}
\thanks{}

%    author three information
\author{Jean-Dominique Deuschel}
\address{Technische Universit\"at Berlin}
\curraddr{Strasse des 17. Juni 136, 10623 Berlin}
\email{deuschel@math.tu-berlin.de}
\thanks{}

%    author four information
\author{Martin Slowik}
\address{Technische Universit\"at Berlin}
\curraddr{Strasse des 17. Juni 136, 10623 Berlin}
\email{slowik@math.tu-berlin.de}
\thanks{}



\date{\today}

\dedicatory{}

\begin{abstract}
  We study a continuous-time random walk, $X$, on $\bbZ^d$ in an environment of dynamic random conductances taking values in $(0, \infty)$.  We assume that the law of the conductances is ergodic with respect to space-time shifts.  We prove a quenched invariance principle for the Markov process $X$ under some moment conditions on the environment.  The key result on the sublinearity of the corrector is obtained by Moser's iteration scheme.
\end{abstract}

\maketitle

\tableofcontents

\section{Introduction}
Random walks in random environment is a topic of major interest in probability theory. A specific model for such a random walks that has been intensively studied during the last decade is the Random Conductance Model (RCM). The question whether a quenched invariance principle or quenched functional central limit theorem (QFCLT) holds is of particular interest.  In the case of an environment generated by static i.i.d.\ random variables this question has been object of very active research (see \cite{ABDH12, Bi11} and references therein). Recently, in \cite{ADS15} a QFCLT has been proven for random walks under general ergodic conductances satisfying a certain moment condition.

Quenched invariance principles have also been shown for various models for random walks evolving in dynamic random environments (see \cite{An14, BZ06, BMP04, DL09, JR11, RV13, RS05}). Here analytic, probabilistic and ergodic techniques were invoked, but assumptions on the ellipticity and the mixing behaviour of the environment remained a pivotal requirement. For instance, the QFCLT for the time-dynamic RCM in \cite{An14} required strict ellipticity, i.e.\ the conductances are almost surely uniformly bounded and bounded away from zero,  as well as polynomial mixing, i.e.\ the polynomial decay of the correlations of the conductances in space and time.
In this paper we significantly relax these assumptions and show a QFLCT for the dynamic RCM with degenerate space-time ergodic conductances that only need to satisfy a moment condition. In contrast to the earlier results mentioned above the environment is \emph{not} assumed to be strictly elliptic or mixing or Markovian in time and we also do not require any regularity with respect to the time parameter.

\subsection{The setting}
Consider the $d$-dimensional Euclidean lattice, $(\bbZ^d, E_d)$, for $d \geq 2$, whose edge set, $E_d$, is given by the set of all non-oriented nearest neighbor bonds, that is $E_d = \{ \{x,y\} :  x,y \in \bbZ^d,\ |x-y| = 1 \}$.  For any $A \subset \bbZ^d$ we denote by $|A|$ the cardinality of the set $A$. Further, we denote by $B(x,r) \ldef \{y \in \bbZ^d \,:\, d(x,y) \leq \lfloor r \rfloor\}$ the closed ball with center $x$ and radius $r$ with respect to the natural graph distance $d$, and we write $B(r) \ldef B(0,r)$.  We also write $B_r$, $r > 0$, for closed balls in $\bbR^d$ with respect to the $\ell^1(\bbR^d)$-norm with center at the origin and radius $r$.  The canonical basis vectors in $\bbR^d$ will be denoted by $e_1, \ldots, e_d$.

The graph $(\bbZ^d, E_d)$ is endowed with time-dependent positive weights, that is, we consider a family $\om = \{\om_t(e) : e \in E_d, \, t\in \bbR \} \in \Om \ldef (0, \infty)^{\bbR \times E_d}$.  We refer to $\om_t(e)$ as the \emph{conductance} on an edge $e$ at time $t$. To simplify notation,  for $x,y\in \bbZ^d$ and $t\in \bbR$ we set $\om_t(x,y) = \om_t(y,x) = \om_t(\{x,y\})$ if $\{x,y\} \in E_d$ and $\om_t(x,y) = 0$ otherwise.  A \emph{space-time shift} by $(s,z) \in \bbR\times \bbZ^d$ is a map $\tau_{s,z}\!: \Om \to \Om$ defined by
\begin{align}
  \big(\tau_{s,z}\, \om \big)_t(\{x,y\})
  \;\ldef\;
  \om_{t+s}(\{x+z, y+z\}),
  \qquad
  \forall\,t \in \bbR,\; \{x,y\} \in E_d.
\end{align}
The set $\{\tau_{t,x} : x \in \bbZ^d, t \in \bbR \}$ together with the operation $\tau_{t,x} \circ \tau_{s,y} \ldef \tau_{ t+s,x+y}$ defines the \emph{group of space-time shifts}.

Finally, let $\Om$ be equipped with a $\si$-algebra, $\cF$, and a probability measure, $\prob$, so that $(\Om, \cF, \prob)$ becomes a probability space.  We also write $\mean$ to denote the expectation with respect to $\prob$.
\begin{assumption}\label{ass:P}
  Assume that $\prob$ satisfies the following conditions:
  \begin{enumerate}[(i)]
  \item $\mean\big[\om_t(e)\big]< \infty$ and $\mean\big[\om_t(e)^{-1}\big] < \infty$ for all $e \in E_d$ and $t \in \bbR$.
  \item $\prob$ is ergodic and stationary with respect to space-time shifts, that is $\prob \circ\, \tau_{t,x}^{-1} \!= \prob\,$ for all $x \in \bbZ^d$, $t\in \bbR$, and $\prob[A] \in \{0,1\}\,$ for any $A \in \cF$ such that $\prob[A \triangle \tau_{t,x}(A)] = 0\,$ for all $x \in \bbZ^d$, $t \in \bbR$.
  \item For every $A\in \mathcal{F}$ the mapping $(\om,t,x)\mapsto \indicator_A(\tau_{t,x}\om)$ is jointly measurable with respect to the $\sigma$-algebra $\mathcal{F}\otimes \mathcal{B}(\bbR)\otimes \mathcal{P}(\bbZ^d)$.
  \end{enumerate}
\end{assumption}
\begin{remark}
  (i) Note that Assumption~\ref{ass:P}(i) implies that $\prob\!\big[0 < \om_t(e) < \infty\big] = 1$ for all $e \in E_d$ and almost all $t \in \bbR$.

  (ii) The static model where the conductances are constant in time and ergodic with respect to space shifts is included as a special case.

  (iii) Under Assumption~\ref{ass:P} we have the following version of the ergodic theorem (see e.g.\ \cite[Chapter~6.2]{Kr85}). For any $\vp \in L^1(\Om, \prob)$,
  \begin{align}\label{eq:ergodic:time-space}
    \lim_{n \to \infty}
    \frac{1}{n^2}\,
    \int_0^{n^2}\!
      \frac{1}{|B(n)|} \sum_{x\in B(n)} \!\varphi(\tau_{t,x} \om)\,
    \md t
    \;=\;
    \mean\!\big[\vp\big]
    \qquad \prob\text{-a.s and in } L^1(\Om, \prob).
  \end{align}
\end{remark}
\begin{remark}\label{rem:time_shifts}
  Let $p\geq 1$ and $T_t\!: L^p(\Om, \prob) \to L^p(\Om, \prob)$ be the map defined by $T_t \vp \ldef \vp \circ \tau_{t,0}$.  Then Assumption~\ref{ass:P} (ii) implies that $\{T_t : t \in \bbR \}$ is a strongly continuous contraction group (SCCS) on $L^p(\Om, \prob)$, cf.\ \cite[Section 7.1]{ZKO94} for $p=2$.
\end{remark}
We denote by $D(\bbR, \bbZ^d)$ the space of $\bbZ^d$-valued c\`{a}dl\`{a}g functions on $\bbR$.  We will study the dynamic nearest-neighbour \emph{random conductance model}.  For a given $\om \in \Om$ and for $s \in \bbR$ and $x \in \bbZ^d$, let $\Prob_{s,x}^{\om}$ be the probability measure on $D(\bbR, \bbZ^d)$, under which the coordinate process $(X_t : t \in \bbR)$ is the continuous-time Markov chain on $\bbZ^d$ starting in $x$ at time $t = s$ with time-dependent generator  (in the $L^2$ sense) acting on bounded functions $f\!: \bbZ^d \to \bbR$ as
\begin{align}\label{eq:LV}
  \mathcal{L}_t^\om f(x)
  \;=\;
  \sum_{y \sim x} \om_t(x, y) \big( f(y) \,-\, f(x)\big).
\end{align}
That is, $X$ is the time-inhomogeneous random walk, whose time-dependent jump rates are given by the conductances.  Note that the counting measure, independent of $t$, is an invariant measure for $X$. Further, the total jump rate out of any site $x$ is not normalised, in particular the sojourn time at site $x$ depends on $x$. Therefore, the random walk $X$ is sometimes called the \emph{variable speed random walk (VSRW)}.

\subsection{Main Results}
We are interested in the $\prob$-almost sure or quenched long time behaviour of this process.  Our main objective is to establish a quenched functional central limit theorem for the process $X$ in the sense of the following definition.
\begin{definition} \label{def:QFCLT}
  Set $X_t^{(n)} \ldef \frac{1}{n} X_{n^2 t}$, $t \geq 0$. We say that the \emph{Quenched Functional CLT} (QFCLT) or \emph{quenched invariance principle} holds for $X$ if for $\prob$-a.e.\ $\om$ under $\Prob_{0,0}^{\om}$, $X^{(n)}$ converges in law to a Brownian motion on $\bbR^d$ with covariance matrix $\Si^2 = \Si \cdot \Si^T$.  That is, for every $T > 0$ and every bounded continuous function $F$ on the Skorohod space $D([0,T], \bbR^d)$, setting $\psi_n = \Mean_{0,0}^{\om}[F(X^{(n)})]$ and $\psi_\infty = \Mean_{0,0}^{\mathrm{BM}}[F(\Si \cdot W)]$ with $(W, \Prob_{\!0,0}^{\mathrm{BM}})$ being a Brownian motion started at $0$, we have that $\psi_n \rightarrow \psi_\infty$ $\prob$-a.s.
\end{definition}
As our main result we establish a QFCLT for $X$ under some additional moment conditions on the conductances. In order to formulate this moment condition we first define measures $\mu_t^{\om}$ and $\nu_t^{\om}$ on $\bbZ^d$ by
\begin{align*}
  \mu_t^{\om}(x) \;\ldef\; \sum_{x \sim y}\, \om_t(x,y)
  \qquad \text{and} \qquad
  \nu_t^{\om}(x) \;\ldef\; \sum_{x \sim y}\, \frac{1}{\om_t(x,y)}.
\end{align*}
In addition, for arbitrary numbers $p, p' \geq 1$ and any non-empty compact interval $I \subset \bbR$ and any finite $B \subset \bbZ^d$  let us introduce a space-time averaged $L^{p, p'}$-norm on functions $u\!: \bbR \times \bbZ^d \to \bbR$ by
\begin{align*}
  \Norm{u}{p, p', I \times B}
  \;\ldef\;
  \bigg(
    \frac{1}{|I|}
    \int_{I}
      \Norm{u(t, \cdot)}{p, B(n)}^{p'}
    \md t
  \bigg)^{\!\!1/p'}
  \mspace{-12mu}\ldef\;
  \bigg(
    \frac{1}{|I|}
    \int_{I}
      \bigg( \frac 1 {|B|} \sum_{x \in B} |u(t,x)|^{p} \bigg)^{\!\!p'/p}
    \md t
  \bigg)^{\!\!1/p'}\mspace{-18mu}.
\end{align*}
Note that by Jensen's inequality $ \Norm{u}{p, p', I \times B} \leq  \Norm{u}{q, q', I \times B}$ if $q\geq p$ and $q'\geq p'$.
\begin{assumption}\label{ass:moment}
  There exist  $p, p', q, q' \in (1, \infty]$ satisfying
  \begin{align} \label{eq:condpqprime}
    \frac{1}{p} \cdot \frac{p'}{p'-1} \cdot \frac{q'+1}{q'} \,+\, \frac{1}{q}
    \;<\;
    \frac{2}{d}
  \end{align}
  such that $\prob$-a.s.
  \begin{align}\label{eq:moment_condition}
    \limsup_{n \to \infty} \Norm{\mu^{\om}}{p, p', Q(n)}
    \;<\;
    \infty,
    \qquad
    \limsup_{n \to \infty} \Norm{\nu^{\om}}{q, q',Q(n)} < \infty,
  \end{align}
  where $Q(n)\ldef [0,n^2] \times B(n)$.
\end{assumption}
\begin{remark}
  (i) Assume that for any $x \in \bbZ^d$ with $|x| = 1$,
  \begin{align*}
    \om_0(0,x)
    \;=\;
    \mean\!\big[\om_0(0,x) \mid \cT\big]\,
    \mean\!\big[\om_0(0,x) \mid \cI\big],
  \end{align*}
  where $\cT$ denotes the $\si$-algebra of sets invariant under time-shifts and $\cI$ the $\si$-algebra of sets invariant under space-shifts. Then, a sufficient moment condition for \eqref{eq:moment_condition} to hold is
  \begin{align*}
    \mean\!\Big[\mean[\om_0(0,x)\mid \cT]^p \Big] \;<\; \infty,
    \qquad
    \mean\!\Big[\mean[\om_0(0,x)\mid \cI]^{p'} \Big] \;<\; \infty
  \end{align*}
  and
  \begin{align*}
    \mean\!\Big[\mean[\om_0(0,x)\mid \cT]^{-q}\Big] \;<\; \infty,
    \qquad
    \mean\!\Big[\mean[\om_0(0,x)\mid \cI]^{-q'}\Big] \;<\; \infty.
  \end{align*}
  Indeed, for all $|x| = 1$ the function $f_x(\om) \ldef \mean[\om_0(0,x)\mid \cT] $ is time-invariant and $g_x(\om) \ldef \mean[\om_0(0,x)\mid \cI] $ is space-invariant which yields
  \begin{align*}
    \Norm{\mu^\om}{p, p',Q(n)}
    \;\leq\;
    \sum_{|x| = 1} \Norm{f_x}{p, B(n)}\, \Norm{g_x}{p',[0,n^2]}
    \underset{n \to \infty}{\;\longrightarrow\;}
    \sum_{|x| = 1} \mean\!\big[ f_x^p\big]^{1/p}\, \mean\!\big[ g_x^{p'}\big]^{1/p'}
  \end{align*}
  by the ergodic theorem and similarly for $q,q'$.  In particular, notice that if the measure $\prob$ is space-ergodic we always have
  \begin{align*}
    \mean[\om_0(0,x)\mid \cI]^{p'} \;=\; \mean[\om_0(0,x)]^{p'} \;<\; \infty,
  \end{align*}
  so that we can choose $p'$ and $q'$ to be infinite.

  (ii) Clearly the example in (i) can be made more general by considering conductances which are a mixture of products  $f \cdot g$ where $f$ is time-invariant and $g$ is space invariant. For example let
  \begin{align*}
    \om_0(0,x)
    \;\ldef\;
    \sum_{i=1}^{N} f_{i,x}(\om)\, g_{i,x}(\om),
    \qquad |x|=1,
  \end{align*}
  with $f_{i,x}$ being time-invariant and $g_{i,x}$  space-invariant. In this case for \eqref{eq:moment_condition} to hold one needs to assume that
  \begin{align*}
    \max_{|x|=1,\,  i=1,\ldots, N}
    \left\{
      \mean\!\big[f_{i,x}^p\big],\, \mean\!\big[f_{i,x}^{-q}\big],\,
      \mean\!\big[g_{i,x}^{p'}\big],\,\mean\!\big[g_{i,x}^{-q'}\big]
    \right\}
    \;<\; \infty.
  \end{align*}

  (iii)  In the case $p' = p$ and $q'=q$ Assumption~\ref{ass:moment} directly translates into a moment condition, which does not involve any conditioning on invariant sets.  More precisely, if there exist $p, q \in (1, \infty]$ satisfying
  \begin{align*}
    \frac{1}{p-1} \,+\, \frac{1}{(p-1) q} \,+\, \frac{1}{q}
    \;<\;
    \frac{2}{d}
  \end{align*}
  such that
  \begin{align*}
    \mean\!\big[\om_t(e)^p\big] \;<\; \infty
    \quad \text{and} \quad
    \mean\!\big[\om_t(e)^{-q}\big] \;<\; \infty
  \end{align*}
  for any $e \in E_d$ and $t \in \bbR$, then Assumption~\ref{ass:moment} holds by the ergodic theorem.
\end{remark}
\begin{theorem}\label{thm:main}
  Suppose that $d \geq 2$ and Assumptions \ref{ass:P} and \ref{ass:moment} hold.  Then, the QFCLT holds for $X$ with a deterministic non-degenerate covariance matrix $\Si^2$.
\end{theorem}
For the static RCM a QFCLT is proven in \cite{ADS15} for stationary ergodic conductances $\{\om(e), e\in E_d\}$ satisfying $\mean[\om(e)^p]<\infty$ and $\mean [\om(e)^{-q}]<\infty$ for $p,q>1$ such that $1/p+1/q<2/d$. Since in the static case we can choose $p'=q'=\infty$, the moment condition for the static model can be recovered in \eqref{eq:condpqprime}.

In the setting of general ergodic environments it is natural to expect that some moment conditions are needed in view of the results in \cite{BBT16}, where Barlow, Burdzy and Tim\'ar give an example for a static RCM on $\bbZ^2$ for which the QFCLT fails but a weak moment condition is fulfilled.

One motivation to study the dynamic RCM is to consider random walks in an environment generated by some interacting particle systems like zero-range or exclusion processes (cf.\ \cite{DGR15, MO16}). Recently, some  on-diagonal upper bounds  for the transition kernel of  a  degenerate time-dependent conductances model are obtained in \cite{MO16}, where the conductances are uniformly bounded from above but they are allowed to be zero at a a given time satisfying a lower moment condition.  In \cite{HK16} it is shown that for  uniformly elliptic dynamic RCM in discrete time -- in contrast to the time-static case --  two-sided Gaussian heat kernel estimates are not stable under perturbations. In a time dynamic balanced environment a QFCLT under moment conditions has been recently shown in \cite{DGR15}.

An annealed FCLT has been obtained for strictly elliptic conductances in \cite{An14}, for non-elliptic conductances generated by an exclusion process in \cite{Av12} and for a similar one-dimensional model allowing some local drift in \cite{AdSV13} and recently for environments generated by random walks in \cite{HdHdSST15}.  In \cite{BCDG13, Mi16} random walks on the backbone of an oriented percolation cluster are considered, which are interpreted as the ancestral lines in a population model.

Finally, let us remark that there is a link between the time dynamic RCM and Ginzburg-Landau interface models as such random walks appear in the so-called Helffer-Sj\"ostrand representation of the space-time covariance in these models (cf.\ \cite{DD05, An14}). However, in this context the annealed FCLT is relevant.

\subsection{The method}
We follow the most common approach to prove a QFLCT for the RCM and introduce the so-called harmonic coordinates, that is we construct a \emph{corrector} $\chi\!:\Om \times \bbR \times \bbZ^d \to \bbR^d$ such that
\begin{align*}
  \Phi(\om,t, x) \;=\; x - \chi(\om,t, x)
\end{align*}
is a space-time harmonic function. In other words,
\begin{align} \label{eq:harm_coord_intro}
  \partial_t \Phi(\om, t, x) +  \cL_t^\om \Phi(\om,t,x)
  \;=\;
  0.
\end{align}
This can be rephrased by saying that $\chi$ is a solution of the time-inhomogeneous Poisson equation
\begin{align} \label{eq:pois_chi}
  \partial_t u + \cL_t^{\om} u \;=\; \cL_t^{\om} \Pi,
\end{align}
where $\Pi$ denotes the identity mapping on $\bbZ^d$.  Recall that one property of the \emph{static} RCM -- being one its main differences to other models for random walks in random media -- is the reversibility of the random walk w.r.t.\  its speed measure. In our setting, the generator $(\partial_t + \cL^\om_t)$ of the space-time process $(t, X_t)$ is asymmetric and the construction of the corrector as carried out for instance in \cite{ABDH12,Bi11} fails, since it is based on a simple projection argument using the symmetry of the generator and an integration by parts. In \cite{An14} it was possible to construct the corrector by techniques close to the original method by Kipnis and Varadhan, since in the case of strictly elliptic conductances the asymmetric part can be controlled and a sector condition holds. In our degenerate situation, the construction of the corrector is indeed one of the most challenging parts to prove the QFCLT. Following the approach in \cite{FK99}, we first solve a regularised corrector equation by an application of the Lax-Milgram lemma and then we obtain the harmonic coordinates by taking limits in a suitable distribution space. The resulting corrector function consists of two parts, one part $\chi_0$ being time-homogeneous and invariant w.r.t.\ space shifts in the sense that for every fixed $t$ it satisfies $\bbP$-a.s.\ the \emph{cocycle property} (see Definition~\ref{def:cocycle} below) and a second part which is only depending on the time variable and which therefore does not appear in the corrector for the time-static model.

Given the harmonic coordinates as a solution of \eqref{eq:harm_coord_intro} the process
\begin{align*}
  M_t \;=\; X_t - \chi(\om,t, X_t)
\end{align*}
is a martingale under $\Prob_{\!0,0}^{\om}$ for $\prob$-a.e.\ $\om$, and a QFCLT for the martingale part $M$ can be easily shown by standard arguments.  We thus get a QFCLT for $X$ once we verify that $\prob$-almost surely the corrector is sublinear:
\begin{align} \label{eq:sublin_intro}
  \lim_{n \to \infty}  \max_{(t,x)\in Q(n)} \frac{ \left| \chi(\om,t,x) \right|}{n}
  \;=\;
  0.
\end{align}
This control on the corrector implies that for any $T>0$ and $\prob$-a.e\ $\om$,
\begin{align*}
  \sup_{0 \,\leq\, t \,\leq\, T}\,
    \frac{1}{n}\, \Big| \chi\big(\om,n^2 t, n\, X_{t}^{(n)}\big) \Big|
    \;\underset{n \to \infty}{\longrightarrow}\;
    0
    \quad \text{ in $\Prob_{\!0,0}^\om$-probability}
\end{align*}
(see Proposition~\ref{prop:contr_corr} below).  Combined with the QFCLT for the martingale part this gives Theorem~\ref{thm:main}.

Once the corrector is constructed, the remaining difficulty in the proof of the QFCLT is to prove \eqref{eq:sublin_intro}.  In a first step we show that the rescaled corrector converges in the space-time averaged $\Norm{\cdot}{1,1, Q(n)}$-norm to zero (see Proposition~\ref{prop:sublinearity:l1} below). This is based on some input from ergodic theory,  see Section~\ref{sec:corr_sublin} for more details. In a second step we establish a maximal inequality for the corrector as a solution of \eqref{eq:pois_chi} using Moser iteration, that is we show that the maximum of the rescaled corrector in \eqref{eq:sublin_intro} can be controlled by its $\Norm{\cdot}{1,1, Q(n)}$-norm (see Proposition~\ref{prop:maxin_lattice} below). In the case of static conductances Moser iteration has already been implemented in order to show the QFCLT in \cite{ADS15}, but also to obtain a local limit theorem and elliptic and parabolic Harnack inequalities in \cite{ADS16} as well as upper Gaussian estimates on the heat kernel in \cite{ADS16a}. In the present time-inhomogeneous setting involving a time-dependent operator $\cL^\om_t$  a space-time version of the Sobolev inequality in \cite{ADS15} is needed and the actual iteration procedure has to be carried out in both the space and the time parameter of the space-time averaged norm (cf.\ \cite{KK77}).

The paper is organised as follows:  In Section~\ref{sec:harm_coord} we construct the corrector and show some of its properties. Then, in Section~\ref{sec:corr_sublin} we prove the sublinearity of the corrector \eqref{eq:sublin_intro} and complete the proof of the QFCLT in Section~\ref{sec:QFCLT}.  The maximal inequality for the time-inhomogeneous Poisson equation in \eqref{eq:pois_chi} is proven in a more general context in Section~\ref{sec:mos_it}.

Throughout the paper, we write $c$ to denote a positive constant which may change on each appearance.  Constants denoted by $C_i$ will be the same through each argument.

\section{Harmonic embedding and the corrector} \label{sec:harm_coord}
Throughout this section we suppose that Assumption~\ref{ass:P} holds.
\subsection{Setup and Preliminaries}
Let us denote by $\cN \ldef \{x \in \bbZ^d : |x|=1\}$ the set of all neighbours of the origin in $\bbZ^d$.  Further, we endow the space $\Om \times \cN$ with the measure $m$ defined by
\begin{align}
  m(\md \om, \md z)
  \;\ldef\;
  \sum_{x \in \bbZ^d} \om_0(0, x) \prob(\md \om) \otimes \de_x(z).
\end{align}
It is easy to check that $L^2(\Om \times \cN, m)$ is a Hilbert space.  For functions $\phi\!: \Om \to \bbR$ we define the horizontal gradient $\mD \phi\!: \Om \times \bbZ^d \to \bbR$ as $\mD \phi(\om, x) \ldef \phi(\tau_{0,x}\om) - \phi(\om)$.  We will also write $\mD_x \phi(\om)$ for $\mD \phi(\om, x)$ with $x \in \cN$.  Notice that $\mD \phi \in L^2(\Om \times \cN, m)$ for any $ \phi \in L^2(\Om, \prob)$.  Further, we define
\begin{align*}
  L^2_{\mathrm{pot}}
  \;\ldef\;
  \overline{
    \{
      \mD \phi \,:\,
      \phi :\Om \to \bbR\; \text{bounded}
    \}
  }^{\|\cdot\|_{L^2(\Om \times \cN, m)}}
\end{align*}
to be the closure of the set of gradients in $L^2(\Om \times \cN, m)$ and let $L^2_{\mathrm{sol}}$ be its orthogonal complement in $L^2(\Om \times \cN, m)$, i.e.\
\begin{align*}
  L^2(\Om \times \cN, m)
  \;=\;
  L^2_{\mathrm{pot}} \oplus L^2_{\mathrm{sol}}.
\end{align*}
\begin{lemma}[cycle condition]\label{lemma:cycle}
  For any $\psi \in L^2_{\mathrm{pot}}$ and any sequence $(x_0, \ldots, x_k)$ in $\bbZ^d$ with $x_0 = x_k$ and $x_i - x_{i-1} \in \cN$ for all $i$, then $\sum_{i=1}^k \psi(\tau_{0,x_{i-1}\,}\om, x_i-x_{i-1}) = 0$.
\end{lemma}
\begin{proof}
  Follows directly from the definition of $L^2_{\mathrm{pot}}$.
\end{proof}
For any $\psi \in L^2_{\mathrm{pot}}$ we define its extension $\Psi\!: \Om \times \bbZ^d \to \bbR$ in the following way.  For any $0 \ne x \in \bbZ^d$ choose a sequence $(x_0, \ldots, x_k)$ in $\bbZ^d$ in such a way that $x_0 = 0$, $x_k = x$ and $x_i - x_{i-1} \in \cN$ for all $i$ and set
\begin{align}\label{eq:def:extension}
  \Psi(\om, 0)
  \;\ldef\;
  0,
  \qquad \text{and} \qquad
  \Psi(\om, x)
  \;=\;
  \sum_{i=1}^k \psi(\tau_{0,x_{i-1}\,}\om, x_{i}- x_i).
\end{align}
As a consequence of Lemma~\ref{lemma:cycle}, $\Psi$ does not depend on the choice of paths.
\begin{definition} \label{def:cocycle}
  A measurable function $\Psi\!: \Om \times \bbZ^d \to \bbR$, also called random field, satisfies the \emph{cocycle property} (in space), if for $\prob$-a.e. $\om$,
  \begin{align}
    \Psi(\tau_{0,x} \om, y-x )
    \;=\;
    \Psi(\om, y) \,-\, \Psi(\om,x),
    \qquad \forall\, x, y \in \bbZ^d.
  \end{align}
  We denote by $L^2_{\mathrm{cov}}$ the set of function $\Psi\!: \Om \times \bbZ^d \to \bbR$ which satisfies the cocycle property such that
  \begin{align*}
    \Norm{\Psi}{L_\mathrm{cov}^2}^2
    \;\ldef\;
    \mean\!\Big[
      {\textstyle \sum_{x \in \bbZ^d}}\, \om_0(0, x)\, \Psi(\om,x)^2
    \Big]
    \;<\;
    \infty.
  \end{align*}
\end{definition}
Although $|| \cdot ||_{L^2_\mathrm{cov}}$ coincides with the norm on $L^2(\Om \times \cN, m)$, we nevertheless introduce this notation to stress the fact that we apply it to functions $\Psi\!:\Om \times \bbZ^d \to \bbR$ that satisfies in addition the cocycle property.
\begin{lemma}\label{lemma:L2cov}
  Let $\Psi \in L^2_{\mathrm{cov}}$.  Then
  \begin{enumerate}[(i)]
  \item $\Psi(\om, 0) = 0$ and $\Psi(\tau_{0,x} \om, -x) = -\Psi(\om, x)$ for all $x \in \bbZ^d$.
  \item $\|\Psi\|_{L^2_{\mathrm{cov}}} = 0$, if and only if, $\Psi(\om, x) = 0$ $\prob$-a.s.\ for all $x \in \bbZ^d$.
  \end{enumerate}
\end{lemma}
\begin{proof}
  (i) follows immediately from the cocycle property.  (ii) is obvious due to the stationarity of $\prob$ and the fact that $\om_0(e) > 0$ $\prob$-a.s.\ for any $e \in E_d$.
\end{proof}
Recall that, by Remark \ref{rem:time_shifts}, the group $\{T_t\}_{t\in \bbR}$ is a SCCG on $L^2(\Om, \prob)$, therefore it has an infinitesimal generator $\mD_0$, whose domain $\cD(\mD_0)$ is dense in $L^2(\Om, \prob)$,
\begin{align*}
  \mD_0 \phi \;\ldef\; \lim_{h \to 0} \frac{T_h \phi - \phi}{h},
\end{align*}
whenever the limit exists in $L^2(\Om, \prob)$.  Finally, we denote by $\langle \cdot, \cdot\rangle_{L^2(\Om \times \cN, m)}$ and $\langle \cdot, \cdot\rangle_{L^2(\Om,\prob)}$ the scalar product in $L^2(\Om \times \cN, m)$ and $L^2(\Om, \prob)$, respectively.
\begin{lemma} \label{lem:prop_coll}
  \begin{enumerate}[(i)]
  \item The operator $\mD_0$ is antisymmetric in $L^2(\Om, \prob)$, that is
    \begin{align} \label{eq:D0_antisymm}
      \langle \phi, \mD_0\psi\rangle_{L^2(\Om,\prob)}
      \;=\;
      -\langle \mD_0\phi, \psi \rangle_{L^2(\Om,\prob)},
      \qquad \forall\, \phi,\psi \in \cD(\mD_0).
    \end{align}
    In particular $\langle \phi, D_0\phi\rangle_{L^2(\Om,\prob)} = 0$ and $\langle 1, D_0\phi\rangle_{L^2(\Om, \prob)} = 0$.
  \item For every $x \in \bbZ^d$ the operators $D_x$ and $D_0$ commute, that is
    \begin{align} \label{eq:comm_grad}
      \mD_0 \mD_x \phi
      \;=\;
      \mD_x \mD_0 \phi,
      \qquad \forall\, \phi \in \cD(\mD_0).
    \end{align}
  \item For every $x \in \bbZ^d$ the adjoint of the operator $\mD_{x}$ is given by $\mD_{-x}$,
    \begin{align} \label{eq:adj_Dx}
      \langle \phi, \mD_x \psi \rangle_{L^2(\Om,\prob)}
      \;=\;
      \langle \mD_{-x} \phi, \psi \rangle_{L^2(\Om,\prob)},
      \qquad  \forall\, \phi, \psi \in L^2(\Om, \prob).
    \end{align}
  \item For every $\xi \in L^2(\Om, \prob)$ the function $t \mapsto \xi(\tau_{t,0\,} \om)$ belongs to $L^2_{\mathrm{loc}}(\bbR)$  $\prob$-a.e.\ $\om$.
  \item For any $\ze \in C^1(\bbR)$ with compact support, $\phi \in \cD(\mD_0)$ and $\psi \in L^2(\Om, \prob)$,
    \begin{align} \label{eq:ibpf_D0}
      \int_{\bbR}
        \ze(t)\,
        \langle \mD_0 \phi \circ \tau_{-t,0}, \psi \rangle_{L^2(\Om,\prob)}\,
      \md t
      \;=\;
      \int_{\bbR}
        \ze'(t)\, \langle \phi \circ \tau_{-t,0}, \psi \rangle_{L^2(\Om,\prob)}\,
      \md t.
    \end{align}
  \item For any  $\phi \in \cD(\mD_0)$, the function $t \mapsto \phi(\tau_{t,0\,} \om)$ is weakly differentiable $\prob$-almost surely.  In particular
    \begin{align}\label{eq:weakdiff}
      \mD_0 \phi (\tau_{t,0}\om) \;=\; \phi' (\tau_{\cdot,0}\om)(t)
    \end{align}
    for almost all $t$, $\prob$-almost surely.
  \item For every $\xi \in L^2(\Om, \prob)$ and every $\psi \in L^2_{\mathrm{pot}}$,
    \begin{align} \label{eq:ibpf_Dx}
      \langle \psi, \mD \xi \rangle_{L^2(\Om \times \cN, m)}
      \;=\;
      -2\,
      \mean\!\Big[ \xi(\om)\;
        {\textstyle \sum_{x \sim 0\,}} \om_0(0,x)\, \psi(\om,x)
      \Big].
    \end{align}
  \end{enumerate}
\end{lemma}
\begin{proof}
  (i) By the shift-invariance of $\prob$ we have for any $\phi, \psi \in L^2(\Om, \prob)$
  \begin{align*}
    \langle \phi, \mD_0\psi \rangle_{L^2(\Om, \prob)}
    &\;=\;
    \lim_{t \to 0}  t^{-1} \langle \phi, T_t \psi - \psi \rangle_{L^2(\Om,\prob)}
    \\[.5ex]
    &\;=\;
    -\lim_{t\to 0} t^{-1} \langle T_{-t} \phi - \phi, \psi \rangle_{L^2(\Om, \prob)}
    \;=\;
    -\langle \psi,  \mD_0\phi \rangle_{L^2(\Om,\prob)}.
  \end{align*}
  The second statement is trivial.
  \\[.75ex]
  (ii) This follows directly from the linearity of $\mD_0$ as
  \begin{align*}
    \mD_0 \mD_x \phi(\om)
    \;=\;
    \mD_0 \big(\phi(\tau_{0,x}\om)-\phi(\om)\big)
    \;=\;
    \mD_0 \phi(\tau_{0,x} \om) - \mD_0 \phi(\om)
    \;=\;
    \mD_x \mD_0\phi(\om),
  \end{align*}
  where we also used that $\phi \circ \tau_{0,x} \in \cD(\mD_0)$ and $\mD_0(\phi \circ \tau_{0,x})= \mD_0(\phi) \circ \tau_{0,x}$.
 \\[.75ex]
  (iii)  Again by the shift invariance of $\prob$ we have
  \begin{align*}
    \langle \phi, \mD_x \psi \rangle_{L^2(\Om,\prob)}
    &\;=\;
    \langle \phi, \psi \circ \tau_{0,x} - \psi \rangle_{L^2(\Om,\prob)}
    \\[.5ex]
    &\;=\;
    \langle \phi \circ \tau_{0,-x} - \phi, \psi \rangle_{L^2(\Om,\prob)}
    \;=\;
    \langle \mD_{-x} \phi, \psi \rangle_{L^2(\Om,\prob)}.
  \end{align*}
  (iv) For any compact $I \subset \bbR$ and $\xi \in L^2(\Om, \prob)$
  \begin{align*}
    \mean\!\bigg[ \int_I( \xi \circ \tau_{t,0})^2\, \md t\bigg]
    \;=\;
    \int_I \mean\!\big[(\xi \circ \tau_{t,0})^2\big]\, \md t
    \;=\;
    |I|\, \mean\!\big[\xi^2\big]
    \;<\;
    \infty.
  \end{align*}
  Thus, for $\prob$-a.e.\ $\om$,
  \begin{align*}
    \int_I \xi( \tau_{t,0}\, \om)^2\, \md t
    \;<\;
    \infty.
  \end{align*}
  (v) A simple change of variables gives
  \begin{align*}
    &\int_{\bbR}
      \ze(t)\, \langle \mD_0 \phi \circ \tau_{-t,0}, \psi \rangle_{L^2(\Om,\prob)}\,
    \md t
    \\[.5ex]
    &\mspace{36mu}=\;
    \lim_{h \to 0} \frac{1}{h}
    \bigg(
      \int_{\bbR}\!
        \ze(t)\, \langle \phi \circ \tau_{-t+h,0}, \psi \rangle_{L^2(\Om,\prob)}\,
      \md t
      \,-\,
      \int_{\bbR}\!
        \ze(t)\, \langle \phi \circ \tau_{-t,0}, \psi \rangle_{L^2(\Om,\prob)}\,
      \md t
    \bigg)
    \\[.5ex]
    &\mspace{36mu}=\;
    \lim_{h \to 0} \frac{1}{h}
    \bigg(
      \int_{\bbR}\!
        \ze(s+h)\, \langle \phi \circ \tau_{-s,0}, \psi \rangle_{L^2(\Om,\prob)}\,
      \md s
      \,-\,
      \int_{\bbR}\!
        \ze(s)\, \langle \phi \circ \tau_{-s,0}, \psi \rangle_{L^2(\Om,\prob)}\,
      \md s
    \bigg)
    \\[.5ex]
    &\mspace{36mu}=\;
    \int_{\bbR}
      \ze'(s)\, \langle \phi \circ \tau_{-s,0}, \psi \rangle_{L^2(\Om,\prob)}\,
    \md s.
  \end{align*}
  (vi) It follows by (iv) that $t \mapsto \phi(\tau_{t,0} \om)$ and $t \mapsto \mD_0\phi(\tau_{t,0}\om)$ belong to $L^2_{\mathrm{loc}}(\bbR)$ $\prob$-almost surely. By definition of weak differentiability, it suffices to show that for $\prob$-a.e.\ $\om$ and all $\ze \in C_0^\infty(\bbR)$
  \begin{align} \label{eq:auxiliary}
    \int_{\bbR}
      \ze(t)\, \mD_0 \phi \circ \tau_{t,0}\,
    \md t
    \;=\;
    -\int_{\bbR} \ze'(t)\phi \circ \tau_{t,0} \md t.
  \end{align}
  By Fubini's theorem and the fact that (v) holds for all $\psi\in L^2(\Om, \prob)$, \eqref{eq:auxiliary} follows for any fixed $\ze$ $\prob$-a.s. The null-set where \eqref{eq:auxiliary} does not hold may depend on $\ze$. We can remove this ambiguity using that $C_0^\infty(\bbR)$ is separable.
  \\[.75ex]
  (vii) By the shift invariance of $\prob$ we have for any $\psi \in L^2_{\mathrm{pot}}$
  \begin{align*}
    &\langle \psi, \mD \xi \rangle_{L^2(\Om \times \cN, m)}
    \\[.5ex]
    % &\mspace{36mu}=\;
    % \mean\Big[
    %   {\textstyle \sum_{x \in \bbZ^d\,}} \om_0(0,x)\, \psi(\om, x)\,
    %   \big(\xi (\tau_{0,x}\om) -\xi(\om) \big)
    % \Big]
    % \\[.5ex]
    &\mspace{36mu}=\;
    \sum_{x \in \bbZ^d}
    \Big(
      \mean\big[ \om_0 (-x,0)\, \psi(\tau_{0,-x} \om, x)\, \xi(\om) \big]
      \,-\,
      \mean\big[ \om_0 (0,x)\, \psi(\om, x)\, \xi(\om) \big]
    \Big)
    \\
    &\mspace{36mu}=\;
    \mean\Big[
      {\textstyle \sum_{x \in \bbZ^d\,}} \om_0(0,x)\,
      \big( \psi(\tau_{0,x}\om, -x) \,-\, \psi(\om, x) \big)\, \xi(\om)
    \Big].
  \end{align*}
  Since Lemma~\ref{lemma:cycle} implies that $\psi(\tau_{0,x}\om, -x) = -\psi(\om, x)$ for all $x \in \cN$, the assertion follows.
\end{proof}

\subsection{Construction of the corrector}
In this subsection we construct the corrector.  We introduce the position field $\Pi\!: \Om \times \bbZ^d \to \bbR^d$ with $\Pi(\om, x) = x$.  We  write $\Pi^j$ for the $j$-th coordinate of $\Pi$.   Obviously, $\Pi^j$ satisfies the cocycle property since $\Pi^j(\om, y-x) = \Pi^j(\om, y) - \Pi^j(\om, x)$.  Moreover, for every $j \in \{1, \ldots, d\}$,
\begin{align*}
  \|\Pi^j\|_{L^2_{\mathrm{cov}}}^2
  \;=\;
  \mean\!\Big[{\textstyle \sum_{x \in \bbZ^d\,}} \om_0(0, x)\, |x^j|^2\Big]
  \;\leq\;
   \mean[\mu_0^{\om}(0)]
  \;<\;
  \infty.
\end{align*}
Next, we state the main result of this subsection.
\begin{theorem}\label{thm:harm_coord}
  Suppose that Assumption~\ref{ass:moment} holds.  Then, there exists a function $\Phi_0 = (\Phi_0^1, \ldots, \Phi_0^d)\!: \Om \times \bbZ^d \to \bbR^d$ with $\Phi_0^j \in L_{\mathrm{cov}}^2$ for $j \in \{1, \ldots, d\}$ such that the following hold.
  \begin{enumerate}[(i)]
  \item  For all $j \in \{1, \ldots, d\}$,
    \begin{align}
      \chi_0^j\!: \Om \times \bbZ^d \to \bbR,
      \qquad
      \chi_0^j \ldef \Pi^j - \Phi_0^j
    \end{align}
    is the unique extension of a function in $L^2_{\mathrm{pot}}$.
  \item  The function $\Phi\!: \Om \times \bbR \times \bbZ^d \to \bbR^d$,
    \begin{align}\label{eq:harm_coord2}
      \Phi(\om, t, x)
      \;=\;
      \Phi_0(\tau_{t,0\,} \om, x)
      \,-\,
      \int_0^t \big(\cL_s^{\om} \Phi_0(\tau_{s,0\,}\om, \cdot)\big)(0)\, \md s
    \end{align}
     also called harmonic coordinate, is (time-space) harmonic in the sense that $\Phi$ is differentiable for almost every $t \in \bbR$ and
    \begin{align} \label{eq:harm_coord}
      \partial_t \Phi(\om,t,x) \,+\, \cL^\om_t \Phi(\om,t,x) \;=\; 0,
      \qquad
      \Phi(\om, 0, 0) \;=\; 0.
    \end{align}
  \item The harmonic coordinates $\Phi$ have the asymptotics
    \begin{align*}
      \lim_{n \to \infty} \max_{(t,x) \in Q(n)}
      \frac{1}{n} \big|\Phi(\om, t, x) \,-\, x \big|
      \;=\;
      0.
    \end{align*}
  \end{enumerate}
\end{theorem}
\begin{remark}
  Notice that the harmonic coordinate, as defined above, satisfies the cocycle property (in time-space), that is for $\prob$-a.e.\ $\om$,
  \begin{align*}
    \Phi(\om, t+s, x+y) - \Phi(\om, t, x)
    \;=\;
    \Phi(\tau_{t,x\,}\om, s, y)
  \end{align*}
  for all $s,t \in \bbR$ and $x, y \in \bbZ^d$.  Indeed, since $\Phi_0^j \in L^2_{\mathrm{cov}}$ for any $j \in \{1, \ldots, d\}$, we deduce from \eqref{eq:harm_coord2} that $\cL_r^{\tau_{t,x} \om} \Phi_0(\tau_{t+r,x\,}\om, \cdot)(0) = \cL^{\om}_{r+t} \Phi(\om, t+r, \cdot)(x)$. Hence,
  \begin{align*}
    \Phi(\tau_{t,x\,}\om, s, y)
    &\;=\;
    \Phi_0(\tau_{t+s, x\,}\om, y) 
    - \int_t^{t+s} \big(\cL^{\om}_{r} \Phi(\om, r, \cdot)\big)(x)\, \mathrm{d}r
    \\[.5ex]
    &\overset{\!\!\!\eqref{eq:harm_coord}\!\!\!}{\;=\;}
    \Phi_0(\tau_{t+s, 0\,}\om, x+y) - \Phi_0(\tau_{t+s, 0\,}\om, x) + \Phi(\om, t+s, x) - \Phi(\om, t, x)
    \\[.5ex]
    &\overset{\!\!\!\eqref{eq:harm_coord2}\!\!\!}{\;=\;}
    \Phi(\om, t+s, x+y) - \Phi(\om, t, x).
  \end{align*}
\end{remark}
Before we prove Theorem~\ref{thm:harm_coord} we define the corrector and collect some of its properties.
\begin{definition}
  The corrector $\chi = (\chi^1, \ldots, \chi^d)\!: \Om \times \bbR \times \bbZ^d \to \bbR^d$ is defined as
  \begin{align*}
    \chi(\om, t, x) \;\ldef\; \Pi(\om,x) \,-\, \Phi(\om,t,x).
  \end{align*}
\end{definition}
\begin{corro}\label{corro:property:chi}
  Let $\chi_0^j$ be defined as in the previous theorem and set $\chi_0 = (\chi_0^1, \ldots, \chi_0^d)$.
  \begin{enumerate}[(i)]
  \item $\chi_0^j \in L^1(\prob)$ with $\mean[\chi_0^j(\om, x)] = 0$ for all $|x|=1$.
  \item For $\prob$-a.e.\ $\om$, $t \in \bbR$ and $x \in \bbZ^d$, the corrector can be written as
    \begin{align}\label{eq:corr_decomp}
      \chi(\om, t, x)
      \;=\;
      \chi_0(\tau_{t,0\,}\om, x)
      \,+\,
      \int_0^t \big(\cL_s^{\om} \Phi_0(\tau_{s,0\,}\om, \cdot)\big)(0)\, \md s.
    \end{align}
  \end{enumerate}
\end{corro}
\begin{proof}
  These are immediate consequences from Theorem~\ref{thm:harm_coord}. Note that \eqref{eq:corr_decomp} follows from (ii) since $\chi_0(\tau_{t,0}\om, 0) = 0$ by Lemma~\ref{lemma:L2cov}~(i).
\end{proof}
The rest of this section is devoted to the construction of the harmonic coordinates and the proof of Theorem~\ref{thm:harm_coord} (i) and (ii). Statement (iii) is equivalent to the sublinearity of the corrector and will be proven in Section~\ref{sec:corr_sublin} below.
\vspace{1ex}

Let $\cH^1 \ldef \{\vp \in \cD(\mD_0) \,:\, \mD \vp \in L^2_{\mathrm{pot}}\}$ equipped with the norm given by
\begin{align*}
  \|\vp\|^2_{\cH^1}
  \;\ldef\;
  \|\vp\|_{L^2(\Om,\prob)}^2
  \,+\, \|\mD_0 \vp\|_{L^2(\Om,\prob)}^2 \,+\, \|\mD \vp\|^2_{L^2(\Om \times \cN, m)},
\end{align*}
and a scalar product $\langle \cdot, \cdot \rangle_{\cH^1}$ defined  by polarisation.  It is easy to see that $\cH^1$ is a Hilbert space.  Also, $\cH^1$ is not trivial, since for $\vp \in L^\infty(\Om, \prob)$ and $f \in C^\infty_0(\bbR)$ the function $\vp_{f}\ldef\int_{\bbR} f(s)(\vp \circ \tau_{s,0})\, \md s$ belongs to $\cD(\mD_0) \cap L^\infty(\Om, \prob) \subset \cH^1$.

We want to solve the following equation
\begin{align}\label{eq:correctorequation}
  Q^{\be} (\vp, \xi) \;=\; B^k(\xi),
  \qquad \forall\, \xi \in\cH^1,\ k = {1, \dots, d},
\end{align}
where $B^k(\xi) \ldef \langle \Pi^k, \mD \xi \rangle_{L^2(\Om \times \cN, m)}$ and
\begin{align*}
  Q^{\be} (\vp, \xi)
  &\;\ldef\;
  -2\langle \mD_0 \vp,  \xi\rangle_{L^2(\Om,\prob)}
  + \langle \mD \vp, \mD \xi \rangle_{L^2(\Om \times \cN, m)}
  \\[.5ex]
  &\mspace{72mu}+ \be\, \langle \mD_0 \vp, \mD_0 \xi \rangle_{L^2(\Om, \prob)}
  + \be\, \langle \vp, \xi \rangle_{L^2(\Om,\prob)}.
\end{align*}
\begin{lemma}
  For all $\be > 0$, $Q^{\be}\!: \cH^1 \times \cH^1 \to \bbR$ is a coercive bounded bilinear form, and for all $k = {1, \dots, d}$, $B^k$ is a bounded and linear operator on $\cH^1$.
\end{lemma}
\begin{proof}
  The statement is true basically by definition and Lemma~\ref{lem:prop_coll} (i).  Indeed,
  \begin{align*}
    Q^{\be}(\vp, \vp)
    \;\geq\;
    (1 \wedge \be)\, \|\vp\|^2_{\cH^1}
  \end{align*}
  and by the Cauchy-Schwarz inequality
  \begin{align*}
    |Q^{\be} (\vp, \xi)|
    &\;\leq\;
    % (1 \vee \be)
    % \Big(
    %   2 \| \mD_0 \psi \|^2_{L^2(\Om, \prob)}
    %   \,+\, \|\mD \psi\|^2_{L^2_{\mathrm{cov}}}
    %   \,+\, \|\psi\|^2_{L^2(\Om,\prob)}
    % \Big)^{\!1/2}\!
    % \\
    % &\;\mspace{36mu}\times\;
    % \Big(
    %   2\, \| \xi  \|^2_{L^2(\Om,\prob)}
    %   \,+\, \|\mD \xi\|^2_{L^2_{\mathrm{cov}}}
    %   \,+\, \|\mD_0\xi\|^2_{L^2(\Om,\prob)}
    % \Big)^{\!1/2}
    % \\[.5ex]
    % &\;\leq\;
    (2 + \be)\, \|\psi\|_{\cH^1} \|\xi\|_{\cH^1}.
  \end{align*}
  Similarly, since $ \mean[ \om_0(0,e)] < \infty$ it follows that $B^k$ is bounded for all $k$.
\end{proof}
By an application of Lax-Milgram Lemma it follows that for every $ \be > 0$ there exists $\psi^{\be,k} \in \cH^1$ such that $Q^{\be} (\psi^{\be, k}, \xi) = B^k(\xi)$ holds for all $\xi \in \cH^1$.  In particular, the equation is satisfied for $\xi = \psi^{\be, k}$. We use this information to obtain a first energy bound.
\begin{lemma} \label{lem:energyestimates}
  For all $\be > 0$ and $k = 1, \dots, d$,
  \begin{align}\label{est:energy}
    \| \mD \psi^{\be,k} \|^2_{L^2(\Om \times \cN, m)}
    + \be\, \| \mD_0 \psi^{\be,k} \|^2_{L^2(\Om,\prob)}
    + \be\, \| \psi^{\be,k} \|^2_{L^2(\Om,\prob)}
    \;\leq\;
    \mean\!\big[\mu_0^{\om}(0)\big].
  \end{align}
  Moreover, for $\be \in (0,1]$ and all $k = 1, \dots, d$,
  \begin{align}\label{est:linear_functional}
    \big| \langle  \mD_0 \psi^{\be,k}, \xi \rangle_{L^2(\Om,\prob)}\big|
    \;\leq\;
    2\, \mean\!\big[\mu_0^{\om}(0)\big]^{1/2}\, \|\xi\|_{\cH^1}.
  \end{align}
\end{lemma}
\begin{proof}
  Since $\langle \psi^{\be,k}, \mD_0 \psi^{\be,k} \rangle_{L^2(\Om,\prob)} = 0$ we get from $Q^{\be} (\psi^{\be,k},\psi^{\be,k}) = B^k(\psi^{\be,k})$ that
  \begin{align} \label{eq:phi_beta_k}
    &\| \mD \psi^{\be,k} \|_{L^2(\Om \times \cN, m)}^2
    +  \be\, \| \mD_0 \psi^{\be,k} \|^2_{L^2(\Om, \prob)}
    + \be\, \| \psi^{\be,k} \|^2_{L^2(\Om, \prob)}
    % \nonumber\\[.5ex]
    % &\mspace{36mu}=\;
    % \langle \Pi^k, \mD \psi^{\be,k} \rangle_{L^2_{\mathrm{cov}}}
    \nonumber\\[.5ex]
    &\mspace{36mu}\leq\;
    \mean\!\big[ \mu_0^{\om}(0)\big]^{1/2}\,
    \|\mD \psi^{\be, k}\|_{L^2(\Om \times \cN, m)},
  \end{align}
  where we used the Cauchy-Schwarz inequality in the last step.  By dropping the positive terms with $\be$ in front, we obtain
  \begin{align} \label{eq:energy_estimate}
    \| \mD \psi^{\be,k} \|_{L^2(\Om \times \cN, \prob)}
    \;\leq\;
    \mean\!\big[\mu_0(0)\big]^{1/2}.
  \end{align}
  By combining this with \eqref{eq:phi_beta_k} we immediately get \eqref{est:energy}.

  In order to prove \eqref{est:linear_functional} we use \eqref{eq:correctorequation}, the triangle inequality and the Cauchy-Schwarz inequality to obtain that for any $\be \in (0, 1]$,
  \begin{align*}
    &2\, \big| \langle \mD_0 \psi^{\be,k}, \xi \rangle_{L^2(\Om, \prob)}\big|
    \\[.5ex]
    &\mspace{16mu}\leq\;
    \big| \langle \mD \psi^{\be,k}, \mD \xi \rangle_{L^2(\Om \times \cN, m)} \big|
    +
    \big| \langle \mD_0 \psi^{\be,k}, \mD_0\xi \rangle_{L^2(\Om, \prob)}\big|
    +
    \big| \langle \psi^{\be,k}, \xi \rangle_{L^2(\Om,\prob)}\big|
    +
    |B^k(\xi)|
    \\[.5ex]
    &\mspace{16mu}\leq\;
    \bigg(
      \| \mD \psi^{\be,k} \|_{L^2(\Om \times \cN, m)}
      + \| \mD_0 \psi^{\be,k} \|_{L^2(\Om,\prob)}
      + \| \psi^{\be,k} \|_{L^2(\Om,\prob)}
      + \sqrt{\mean[\mu_0(0)]}
    \bigg)\, \|\xi\|_{\cH^1}.
  \end{align*}
  In view of \eqref{est:energy}, the desired bound \eqref{est:linear_functional} follows.
\end{proof}

By Lemma~\ref{lem:energyestimates} we have that $\mD \psi^{\be,k}$ are uniformly bounded in $L^2(\Om \times \cN, m )$.  Therefore, there exist $\psi^k \in L^2(\Om \times \cN, m)$ such that weakly in $L^2(\Om \times \cN, m)$ along a subsequence $\be \downarrow 0$
\begin{align*}
  \mD \psi^{\be, k} \;\rightharpoonup\; \psi^k.
\end{align*}
In fact $\psi^k \in L^2_{\mathrm{pot}}$, since for all $\xi \in L^2_{\mathrm{sol}}$ we have that $\langle \mD \psi^{\be,k},\xi\rangle_{L^2(\Om \times \cN, m)} = 0$ for all $\be > 0$ and $\langle \psi^k, \xi\rangle_{L^2(\Om \times \cN, m)} = \lim_{\be \to 0} \langle \mD \psi^{\be,k},\xi\rangle_{L^2(\Om \times \cN, m)}$.

As a further consequence of Lemma~\ref{lem:energyestimates} we observe that the  linear functional $F^{\be,k}\!: \cH^1 \to \bbR$ defined by
\begin{align*}
  F^{\be,k}(\xi) \;\ldef\; -\langle \mD_0 \psi^{\be,k} , \xi\rangle_{L^2(\Om, \prob)}
\end{align*}
are uniformly bounded in $\cH^{-1}$, the dual of $\cH^1$.  It follows that there exist $F^k \in \cH^{-1}$ such that weakly in $\cH^{-1}$ along a subsequence $\be \downarrow 0$
\begin{align*}
  F^{\be,k} \;\rightharpoonup\; F^k.
\end{align*}
Recall that weak convergence in $\cH^{-1}$ implies that $F^{\be,k}(\xi) \to F^k(\xi)$ for all $\xi\in \cH^1$.  Thus, by taking the limit in \eqref{eq:correctorequation} as $\be \to 0$ along some subsequence we get
\begin{align}\label{eq:abstracteq}
  2\, F^k(\xi )
  \,+\,
  \langle \psi^k, \mD \xi \rangle_{L^2(\Om \times \cN, m)}
  \;=\;
  B^k(\xi),
  \qquad \forall\, \xi \in \cH^1.
\end{align}
The first term on the left of \eqref{eq:abstracteq} is implicit. We want to identify it at least for a class of functions $\xi \in \cH^1$. This is the content of the next lemma.
\begin{lemma}
  Consider the class
  \begin{align*}
    \cH^1_b
    \;\ldef\;
    \big\{
      \xi \in L^\infty(\Om, \prob) \cap \cD(\mD_0)
      \,:\,
      \mD_0\xi \in L^\infty(\Om, \prob)
    \big\}
  \end{align*}
  Then, $\cH^1_b$ is dense in $L^p(\Om, \prob)$ for all $p \geq 1$.  Moreover, for any $\xi\in \cH^1_b$ and $x \in \cN$
  \begin{align}\label{eq:timederivative}
    F^k(\mD_x \xi)
    \;=\;
    \langle \psi^k(\cdot,-x), \mD_0 \xi \rangle_{L^2(\Om, \prob)}.
  \end{align}
\end{lemma}
\begin{proof}
  For the proof of the density it suffices to show that $\cH^1_b$ is dense in $L^\infty(\Om, \prob)$ with respect to the $L^p(\Om, \prob)$-norm, $p \geq 1$.  To this end, consider $f \in C^\infty_0(\bbR; [0,1])$ such that $\int_\bbR f(s)\, \md s = 1$. For  $\vp \in L^\infty(\Om, \prob)$, define $\vp_{\ep} \ldef \ep^{-1} \int_\bbR f(s/\ep) (\vp \circ\tau_{s,0})\, \md s$ and observe that $\vp_{\ep} \in \cH^1_b$.  Finally, by the strong continuity of $T_t$ in $L^p(\Om, \prob)$ (see Remark \ref{rem:time_shifts}) it follows that $\vp_{\ep} \to \vp$ in $L^p(\Om, \prob)$.

  For the second part of the statement, by \eqref{eq:D0_antisymm}, \eqref{eq:comm_grad} and \eqref{eq:adj_Dx} we get for  $\xi \in \cH^1_b$,
  \begin{align}\label{eq:relationderivative}
    -\langle \mD_0 \psi^{\be,k},  \mD_x \xi \rangle_{L^2(\Om, \prob)}
    &\;=\;
    \langle \psi^{\be,k},  \mD_x  \mD_0 \xi \rangle_{L^2(\Om, \prob)}
    \;=\;
    \langle \mD_{-x}\psi^{\be,k}, \mD_0 \xi \rangle_{L^2(\Om, \prob)}
    \nonumber \\[.5ex]
    &\;=\;
    \big\langle
      \om_0(-x,0) \mD_{-x} \psi^{\be,k} , \om_0(-x,0)^{-1}\mD_0 \xi
    \big\rangle_{L^2(\Om, \prob)}
    \nonumber\\[.5ex]
    &\;=\;
    \langle \mD \psi^{\be,k}, \Xi^x \rangle_{L^2(\Om \times \cN, m)},
  \end{align}
  where $\Xi^x\!: \Om \times \cN \to \bbR$ is defined by
  \begin{align*}
    \Xi^{x}(\om, y)
    \;\ldef\;
    \indicator_{-x}(y)\,
%    \om_0(-x,0)^{-1}\, \mD_0 \xi(\om).
    \frac{\mD_0 \xi(\om)}{\om_0(-x,0)}.
  \end{align*}
  Observe that, since $\xi \in \cH^1_b$, $\mD_x \xi \in \cH^1$ and $\Xi^x \in L^2(\Om \times \cN, m)$, since by Assumption \ref{ass:P}~(i) $\mean[\om_0(0,x)^{-1}] < \infty$.  Using the weak convergence along a subsequence as $\be \downarrow0$ in \eqref{eq:relationderivative} we finally get
  \begin{align*}
    F^k(\mD_x \xi)
    &\;=\;
    \langle \psi^k, \Xi^x \rangle_{L^2(\Om \times \cN, m)}
    % \\[.5ex]
    % &\;=\;
    % \big\langle
    %   \om_0(-x,0) \Psi^k(\cdot,-x), \om_0(-x,0)^{-1}\mD_0 \xi
    % \big\rangle_{L^2(\Om, \prob)}
    % %\\[.5ex]
    \;=\;
    \langle \psi^k(\cdot,-x), \mD_0 \xi \rangle_{L^2(\Om, \prob)},
  \end{align*}
  which is the claim.
\end{proof}
\begin{proof}[Proof of Theorem~\ref{thm:harm_coord} (i) and (ii).]
  In view of \eqref{eq:abstracteq} and \eqref{eq:timederivative}, we obtain for any $\xi \in\cH^1_b$ and $x \in \cN$
  \begin{align*}
    2\, \langle \psi^k(\cdot,x), \mD_0 \xi \rangle_{L^2(\Om, \prob)}
    \,+\,
    \langle \psi^k, \mD \mD_{-x} \xi \rangle_{L^2(\Om \times \cN, m)}
    \;=\;
    B^k(\mD_{-x} \xi),
  \end{align*}
  which can be rewritten as
  \begin{align}\label{eq:tothesolution}
    2\, \langle \psi^k(\cdot, x), \mD_0 \xi \rangle_{L^2(\Om, \prob)}
    \,+\,
    \langle \psi^k - \Pi^k, \mD \mD_{-x} \xi \rangle_{L^2(\Om \times \cN, m)}
    \;=\;
    0.
  \end{align}
  Since $\psi^k \in L^2_{\mathrm{pot}}$, there exists an unique extension $\Psi^k\!: \Om \times \bbZ^d \to \bbR$ that is defined by the formula \eqref{eq:def:extension}.  Moreover, we define
  \begin{align}
    \Phi_0^k\!: \Om \times \bbZ^d \to \bbR,
    \qquad
    \Phi_0^k(\om, x) \;\ldef\; x^k - \Psi^k(\om, x).
  \end{align}
  Obviously, $\Phi_0^k \in L^2_{\mathrm{cov}}$ by construction.  Thus, by the cocycle property (in space)
  \begin{align*}
    \nabla_y \Phi_0^k(\tau_{t,0\,} \om, x)
    \;\ldef\;
    \Phi_0^k(\tau_{t,0\,} \om, x+y) - \Phi_0^k(\tau_{t,0\,} \om, x)
    % \;=\;
    % \Phi_0^k(\tau_{t,x\,}\om, y)
    \;=\;
    y^k - \psi^k(\tau_{t,x\,}\om, y),
  \end{align*}
  for all $t \in \bbR$, $x \in \bbZ^d$ and $y \in \cN$.  Using $\langle \Pi^k, \mD_0 \xi \rangle_{L^2(\Om, \prob)} = 0$, we rewrite \eqref{eq:tothesolution} as
  \begin{align} \label{eq:tothesolution2}
    2\, \langle \nabla_{x} \Phi_0^k(\cdot, 0),  \mD_0 \xi \rangle_{L^2(\Om, \prob)}
    \,+\,
    \langle
      \Phi_0^k,  \mD \mD_{-x} \xi
    \rangle_{L^2(\Om \times \cN, m)}
    \;=\;
    0.
  \end{align}
  Notice that $\xi \circ \tau_{-t,-z} \in \cH^1_b$ for all $\xi \in \cH^1_b$ and $z \in \bbZ^d$, $t \in \bbR$.  Thus, we can replace $\xi$ by $\xi \circ \tau_{-t,-z}$ in \eqref{eq:tothesolution2}, integrate with respect to $t$ against a function $\ze \in C^1(\bbR)$ with compact support and use \eqref{eq:ibpf_Dx} and \eqref{eq:adj_Dx} to obtain
  \begin{align*}
    \int_{\bbR}\!
      \ze(t)
      \Big(
        \langle
          \nabla_{x} \Phi_0^k(\cdot, 0), \mD_0 (\xi \circ \tau_{-t,-z})
        \rangle_{L^2(\Om, \prob)}
        -
        \langle
           \mD_x (\cL_0^{\om} \Phi_0^k)(0), \xi \circ \tau_{-t,-z}
        \rangle_{L^2(\Om, \prob)}
      \Big)
    \md t
    =
    0.
  \end{align*}
  Further, by applying \eqref{eq:ibpf_D0}, Fubini's theorem and the shift invariance of $\prob$,
  \begin{align*}
    \mean\biggl[
      \xi
      \int_{\bbR}
        \Big(
          \ze'(t)\, \nabla_{x} \Phi_0^k(\cdot, 0) \circ \tau_{t,z}
          \,-\,
          \mD_{x} \big( (\cL^\om_0 \Phi_0^k)(0) \circ \tau_{t,z} \big)\,
          \ze(t)
        \Big)\,
      \md t
    \biggr]
    \;=\;
    0.
  \end{align*}
  Since $\mD_x (\cL_0^{\om} \Phi_0^k)(0) \circ \tau_{t,z} = \nabla_x (\cL_t^{\om} \Phi_0^k(\tau_{t,0\,} \om, \cdot))(z)$ and $\cH^1_b$ is dense in $L^p(\Om, \prob)$ for all $p \geq 1$, the equation above implies that
  \begin{align*}
    \nabla_{x}
    \biggl(
      \int_{\bbR}
        \ze'(t)\, \Phi_0^k(\tau_{t,0\,}\om, z) \,-\,
        (\cL^\om_t \Phi_0^k(\tau_{t,0\,}\om, \cdot))(z)\, \ze(t)\,
      \md t
    \biggr)
    \;=\;
    0
  \end{align*}
  for all $y, z\in \bbZ^d$ and  all $\ze \in C_0^{\infty}(\bbR)$, $\prob$-a.s.  In particular, the term in brackets is constant in $z$ and since $\Phi_0(\om,t,0)=0$ we get that
  \begin{align}\label{eq:presol}
    \int_{\bbR}
      -\ze'(t)\, \Phi^k_0(\tau_{t,0\,}\om, z) \,+\,
      (\cL^{\om}_t \Phi_0^k(\tau_{t,0\,}\om, \cdot))(z)\, \ze(t)\,
    \md t
    \;=\;
    \int_{\bbR} (\cL^\om_t \Phi_0^k(\tau_{t,0\,}\om, \cdot))(0)\, \ze(t)\, \md t.
  \end{align}
  From this equation it follows in particular that $t \mapsto \Phi_0^k(\tau_{t,0\,}\om, z)$ is weakly differentiable in time, hence by Sobolev's embedding it is also absolutely continuous in time for all $x \in \bbZ^d$, $\prob$-a.s.\ and differentiable for almost all $t \in \bbR$. In particular $\Phi_0^k(\tau_{t,0\,} \om,z) - \Phi_0^k(\om, z) = \int_0^t \partial_t \Phi_0^k(\tau_{s,0\,}\om, z)\, \md s$ and for almost all $t \in \bbR$, all $z \in \bbZ^d$
  \begin{align*}
    \partial_t \Phi^k_0(\tau_{t,0\,} \om, z)
    \,+\, (\cL^\om_t \Phi_0^k(\tau_{t,0\,}\om, \cdot))(z)
    \;=\;
    (\cL^\om_t \Phi_0^k(\tau_{t,0\,} \om, \cdot))(0).
  \end{align*}
  We define
  \begin{align*}
    \Phi^k(\om, t, z)
    \;\ldef\;
    \Phi_0^k(\tau_{t,0\,} \om, z)
    \,-\, \int_0^t (\cL^\om_s \Phi_0^k(\tau_{s,0\,} \om, \cdot))(0)\, \md s.
  \end{align*}
  Using \eqref{eq:presol} it is easy to see that $\Phi^k$ solves \eqref{eq:harm_coord}. We postpone the proof of (iii) to Proposition \ref{prop:sublin_corr} below.
\end{proof}

\section{Sublinearity of the corrector} \label{sec:corr_sublin}
The key ingredient in the proof of Theorem~\ref{thm:main} is the sublinearity of the corrector as stated in the following proposition, which we prove as the main result in this section.
\begin{prop} \label{prop:sublin_corr}
  Let $d \geq 2$ and suppose that Assumptions \ref{ass:P} and \ref{ass:moment} hold.  Then,
  \begin{align} \label{eq:sublin_corr}
    \lim_{n \to \infty} \max_{(t,x) \in Q(n)} \frac{ | \chi(\om, t, x)|} n
    \;=\;
    0,
    \qquad \prob\text{- a.s.}
  \end{align}
\end{prop}
The proof is based on both ergodic theory and purely analytic tools. First we state the maximum inequality, which we establish in a more general context in Section~\ref{sec:mos_it} below, to bound from above the maximum of the rescaled corrector in $Q(n)$ in terms of its $\Norm{\cdot }{1,1,Q(n)}$-norm.
\begin{prop} \label{prop:maxin_lattice}
  Let $p, p', q, q' \in [1, \infty)$ be as in Assumption~\ref{ass:moment}.  Then, there exist $\ga \equiv \ga(d, p, p'\!, q, q') \in (0,1]$, $\ka \equiv \ka(d, p, p'\!, q, q') > 0$ and $c \equiv c(p, q, q', d) < \infty$ such that for any $j \in \{1, \ldots, d\}$,
  \begin{align*} %\label{eq:maxin_lattice}
    &\max_{(t,x) \in Q(n)}\! n^{-1} \big| \chi^j(\om, t,x) \big|
    \\%[.5ex]  
    &\mspace{36mu}\leq\;
    %\leq   
    c\,
    \Big( \Norm{1\vee \mu^{\om}}{p, p',Q(2n)}
              \Norm{1 \vee \nu^{\om}}{q,q',Q(2n)}
    \Big)^{\!\ka}
    M_{\ga}\big(\Norm{\tfrac{1}{n} \chi^j(\om, \cdot)}{1, 1, Q(2 n)}\big),
  \end{align*}
  where $M_{\ga}(s) \ldef s^{\ga} \vee s$.
\end{prop}
We postpone the proof to Section~\ref{sec:mos_it}.  Proposition~\ref{prop:sublin_corr} is now immediate from Proposition~\ref{prop:maxin_lattice},  Assumption~\ref{ass:moment} and the following proposition.

\begin{prop}\label{prop:sublinearity:l1}
  Suppose $d \geq 2$ and Assumption~\ref{ass:P} holds.  Then for $\prob$-a.e.\ $\om$,
  \begin{align}\label{eq:l1:conv}
    \lim_{n\to \infty}\, \frac{1}{n^{3}}\,
    \int_0^{n^2}\!
      \frac{1}{|B(n)|} \sum_{x \in B(n)}\! \big|\chi(\om, t, x)\big|\,
    \md t
    \;=\;
    0.
  \end{align}
\end{prop}

The rest of this section is devoted to the proof of Proposition~\ref{prop:sublinearity:l1}, which relies on the following two lemmas.  First we recall that the Euclidean lattice $(\bbZ^d, E_d)$ satisfies the classical (strong) $\ell^1$-Poincar\'e inequality
\begin{align}\label{eq:Poincare:l1}
  \sum_{x \in B(n)} \big| u(x) - (u)_{B(n)} \big|
  \;\leq\;
  C_{\mathrm{P}}\, n
  \mspace{-6mu}\sum_{\substack{x, y \in B(n) \\ x \sim y}}\mspace{-10mu}
  \big| u(x) - u(y) \big|
\end{align}
for any function $u\!: \bbZ^d \to \bbR$, where $(u)_{B(n)} \ldef |B(n)|^{-1} \sum_{x \in B(n)} u(x)$; see, for example, \cite[Lemma~3.3.3]{S-C97}.
\begin{lemma}\label{lemma:sublinearity:l1:chi_0}
  Let $d \geq 2$.  Then, for every $j = 1, \ldots, n$ and $\prob$-a.e.\ $\om$,
  \begin{align}\label{eq:sublinearity:l1:chi_0}
    \lim_{n \to \infty} \frac{1}{n^3}\,
    \int_0^{n^2}\!
      \frac{1}{|B(n)|} \sum_{x \in B(n)}
      \Big|
        \chi_{0}^j(\tau_{t, 0\,} \om, x)
        \,-\,
        \big(\chi_{0}^{j}(\tau_{t, 0\,} \om, \cdot) \big)_{B(n)}
      \Big|\,
    \md t
    \;=\;
    0.
  \end{align}
\end{lemma}
\begin{proof}
  Since $\chi_0^j$ is the unique extension of a function $\psi^j \in L^2_{\mathrm{pot}}$, there exists for any $j = 1, \ldots, d$ a sequence of bounded functions $\psi_k^j\!: \Om \to \bbR$ such that $\mD \psi_k^{j} \to \psi^{j}$ in $L^2(\Om \times \cN, m)$ as $k \to \infty$.   Notice that $\chi_0^j(\om, x) = \psi^j(\om, x)$ for all $x \in \cN$ and $\prob$-a.e.\ $\om \in \Om$.  Then, by applying the $\ell^1$-Poincar\'e inequality
  \begin{align*}
    &\frac{1}{n^3}\,
    \int_0^{n^2}\mspace{-12mu}
      \frac{1}{|B(n)|} \sum_{x \in B(n)}\mspace{-3mu}
      \Big|
        \chi_{0}^j(\tau_{t, 0\,} \om, x)
        -
        \big(\chi_{0}^{j}(\tau_{t, 0\,} \om, \cdot) \big)_{B(n)}
      \Big|\,
    \md t
    \\[.5ex]
    &\mspace{36mu}\leq\;
    \frac{1}{n^3}
    \int_0^{n^2}\mspace{-12mu}
    \frac{1}{|B(n)|}
    \sum_{x \in B(n)}\mspace{-3mu}
      \Big|
        (\chi_{0}^j - D\psi_k^j)(\tau_{t, 0\,} \om, x)
        -
        \big( (\chi_{0}^{j} - D\psi_k^j)(\tau_{t, 0\,} \om, \cdot) \big)_{B(n)} 
      \Big|\,
    \md t
    \\
    &\mspace{72mu}+\,
    \frac{4}{n}\, \big\|\psi_k^j\big\|_{L^{\infty}(\Om,\prob)}
    \\[.5ex]
    &\mspace{32mu}\overset{\eqref{eq:Poincare:l1}}{\;\leq\;}
    \frac{C_{\mathrm{P}}}{n^2}\,
    \int_0^{n^2}\mspace{-12mu}
      \frac{1}{|B(n)|} \sum_{\substack{x \in B(n) \\ y \sim 0}}\mspace{-3mu}
      \Big|
        \big(\chi_{0}^j - \mD \psi_k^j\big)(\tau_{t, x\,} \om, y)
      \Big|\,
    \md t
    \,+\,
    \frac{4}{n}\, \big\|\psi_k^j\big\|_{L^{\infty}(\Om,\prob)},
  \end{align*}
  where we used in the second step the cocycle property.  Thus, by the pointwise ergodic theorem \eqref{eq:ergodic:time-space} it follows that for $\prob$-a.e.\ $\om$,
  \begin{align*}
    &\limsup_{n \to \infty} \frac{1}{n^3}\,
    \int_0^{n^2}\mspace{-10mu}
      \frac{1}{|B(n)|} \sum_{x \in B(n)}
      \Big|
        \chi_{0}^j(\tau_{t, 0\,} \om, x)
        \,-\,
        \big(\chi_{0}^{j}(\tau_{t, 0\,} \om, \cdot) \big)_{B(n)}
      \Big|\,
    \md t
    \\[.5ex]
    &\mspace{36mu}\leq\;
    C_{\mathrm{P}}\,
    \mean\!\bigg[
      {\textstyle \sum_{y \sim 0}}
      \Big|
        \big(\chi_0^j - \mD \psi_k^j\big)(\om, y)
      \Big|
    \bigg]
    \;\leq\;
    c\, \mean\!\big[\nu_0^{\om}(0)\big]^{1/2}\,
    \big\|\psi^j - \mD \psi_k^j\big\|_{L^2(\Om \times \cN, m)}.
  \end{align*}
  Since, by construction, $\psi^j - \mD \psi_k^j \to 0$ in $L^2(\Om \times \cN, m)$ as $k \to \infty$, the assertion \eqref{eq:sublinearity:l1:chi_0} follows.
\end{proof}
\begin{lemma}\label{lemma:sublinearity:chi:l1:averaged}
  For every $j = 1, \ldots, n$ and $\prob$-a.e.\ $\om$ we have that
  \begin{align}\label{eq:sublinearity:chi:l1:averaged}
    \lim_{n \to \infty} \frac{1}{n^3}\,
    \int_0^{n^2}
    \mspace{-6mu}
    \frac{1}{|B(n)|} \sum_{x \in B(n)}
      \Big|
        \chi^j(\om, t, x)
        \,-\,
        \big(\chi^j(\om, \cdot, \cdot) \big)_{Q(n)}
      \Big|\,
    \md t
    \;=\;
    0,
  \end{align}
  where $(\chi^j)_{Q(n)}$ denotes the time-space average of the function $\chi^j$ over the time-space cylinder $Q(n) = [0, n^2] \times B(n)$.
\end{lemma}
\begin{proof}
  Consider the function $f\!:\bbR^d \to \bbR$, $x \mapsto \prod_{i=1}^d \hat{f}_i(x_i)$, where $\hat{f}_i \in C_0^{\infty}((-\frac{1}{d}, \frac{1}{d}))$ with $0 \leq \hat{f}_i \leq 1$ and set $f_n(x) \ldef f(x/n)$.  Since $\supp f \in (-\frac{1}{d}, \frac{1}{d})^d$, we have that $\supp f_n \subset B(n)$ for all $n \geq 1$ and $\int_{\mathbb{R}^d} (\partial_y f)(x)\, \md x = 0$ for all $\bbZ^d$ with $|y|=1$, where we denote by $\partial_y f$ the directional derivative of $f$.

  We now address the proof of \eqref{eq:sublinearity:chi:l1:averaged} that comprises two steps.

  \noindent
  \textsc{Step 1:}  Fix some $y \in \bbZ^d$ with $|y|=1$.  Then, for any $\vp \in L^1(\Om, \prob)$ an extension of Birkhoff's theorem, cf.\ \cite[Theorem~3]{BD03}, yields for every $t \in (0,1]$,
  \begin{align*}
    F_n^{\om}(t)
    \;\ldef\;
    \frac{1}{n^{d+2}}\!
    \int_0^{t n^2}\mspace{-8mu}
      \sum_{x \in \bbZ^d} (\partial_y f)(x/n)\, \vp(\tau_{s,x\,} \om)\,
    \md s
    \underset{n \to \infty}{\;\longrightarrow\;}
    t\, \bigg( \int_{\bbR^d} (\partial_y f)(x)\, \md x\bigg) \mean[\vp]
    \;=\;
    0
  \end{align*}
  for $\prob$-a.e.\ $\om$.  In particular, there exists a set $\Om_0 \subset \Om$ having full $\prob$-measure such that for all $\om \in \Om_0$ it holds that $F_n^{\om}(t) \to 0$ as $n \to \infty$ for almost all $t \in [0, 1]$.  Indeed, for
  \begin{align*}
    N
    \;\ldef\;
    \big\{
      (t, \om) \in [0,1] \times \Om \,:\,
      F_n^{\om}(t) \nrightarrow 0 \text{ as } n \to \infty
    \big\}
  \end{align*}
  set $N_t \ldef \{\om \in \Om \,:\, (t, \om) \in N\}$ and $N_{\om} \ldef \{t \in [0,1] \,:\, (t, \om) \in N\}$.  Since $N$ is measurable and $\prob[N_t] = 0$ for every $t \in (0, 1]$, Fubini's theorem implies that $(\mathop{\mathrm{Leb}} \otimes \prob)[N] = 0$.  In particular, for $\prob$-almost all $\om$ it holds that $F_n^{\om}(t) \to 0$ as $n \to \infty$ for a.e.\ $t \in [0,1]$.  Moreover, since for all $t \in [0,1]$,
  \begin{align*}
    F_n^{\om}(t)
    \;\leq\;
    \sup_{n > 0} \frac{1}{n^{d+2}}
    \int_0^{n^2}\mspace{-6mu}
      \sum_{x \in \bbZ^d} \big|(\partial_y f)(x/n)\big|\,
      \big|\vp(\tau_{s,x\,} \om)\big|\,
    \md s
    \;<\;
    \infty
    \qquad \prob\text{-a.s.},
  \end{align*}
  we conclude, by Lebesgue's dominated convergence theorem, that for $\prob$-a.e.\ $\om$,
  \begin{align}\label{eq:sublinearity:l1:step1}
    \lim_{n \to \infty}
    \int_0^1
      \bigg|
        \frac{1}{n^{d+2}}
        \int_0^{t n^2}\mspace{-6mu}
          \sum_{x \in \bbZ^d} (\partial_y f)(x/n)\, \vp(\tau_{s,x\,} \om)\,
        \md s
      \bigg|\,
    \md t
    \;=\;
    0.
  \end{align}
  \textsc{Step 2:} Denote by $( u(t, \cdot) )_{f_n, B(n)}$ the weighted average of a function $u\!: \bbR \times \bbZ^d \to \bbR$,
  \begin{align*}
    \big( u(t, \cdot) \big)_{f_n, B(n)}
    \;\ldef\;
    \frac{c_n}{|B(n)|}
    \sum_{x \in B(n)} \mspace{-3mu} f_n(x)\, u(t,x),
  \end{align*}
  where $c_n \ldef |B(n)| / \big(\sum_{x \in B(n)} f_n(x)\big)$. Set $(u)_{f_n, Q(n)} \ldef \frac{1}{n^2} \int_0^{n^2} \big((u(t, \cdot) \big)_{f_n, B(n)}\, \md t$.  Then, we obtain that
  \begin{align}\label{eq:sublinearity:l1:split}
    &\frac{1}{n^3}
    \int_0^{n^2}\mspace{-6mu}
      \frac{1}{|B(n)|} \sum_{x \in B(n)}
      \Big|
        \chi^j(\om, t, x)
        \,-\,
        \big(\chi^j(\om, \cdot, \cdot) \big)_{Q(n)}
      \Big|\,
    \md t
    \nonumber\\[1ex]
    &\mspace{36mu}\leq\;
    \frac{2}{n^3}
    \int_0^{n^2}\mspace{-6mu}
      \frac{1}{|B(n)|} \sum_{x \in B(n)}
      \Big|
        \chi^j(\om, t, x)
        \,-\,
        \big(\chi^j(\om, \cdot, \cdot) \big)_{f_n,Q(n)}
      \Big|\,
    \md t
    \nonumber\\[1ex]
    &\mspace{36mu}\leq\;
    \frac{2}{n^3}
    \int_0^{n^2}\mspace{-6mu}
      \frac{1}{|B(n)|} \sum_{x \in B(n)}
      \Big|
        \chi^j(\om, t, x)
        \,-\,
        \big(\chi^j(\om, t, \cdot) \big)_{f_n, B(n)}
      \Big|\,
    \md t
    \nonumber\\[1ex]
    &\mspace{72mu}+\,
    \frac{2}{n^3}
    \int_0^{n^2}
      \Big|
        \big(\chi^j(\om, t, \cdot)\big)_{f_n, B(n)}
        \,-\,
        \big(\chi^j(\om, \cdot, \cdot)\big)_{f_n, Q(n)}
      \Big|\,
    \md t
    \nonumber\\[1ex]
    &\mspace{36mu}=\;
    I_1(n) \,+\, I_2(n).
  \end{align}
  Since for any function $u\!:\mathbb{Z}^d \to \mathbb{R}$,
  \begin{align*}
    \sum_{x \in B(n)}\mspace{-6mu} \big|u(x) \,-\, (u)_{f_n, B(n)}\big|
    &\;\leq\;
    \sum_{x \in B(n)}\mspace{-6mu} \big|u(x) \,-\, (u)_{B(n)}\big|
    \,+\,  |B(n)|\, \big|(u)_{f_n,B(n)} \,-\, (u)_{B(n)}\big|
    \\[.5ex]
    &\;\leq\;
    (1+c_n)\, \sum_{x \in B(n)}\mspace{-6mu} \big|u(x) \,-\, (u)_{B(n)}\big|,
  \end{align*}
  where we used the triangular inequality and the fact that $0 \leq f_n \leq 1$, it follows that
  \begin{align*}
    I_1(n)
    \;\leq\;
    \frac{2(1+c_n)}{n^3}
    \int_0^{n^2}\mspace{-6mu}
      \frac{1}{|B(n)|} \sum_{x \in B(n)}
      \Big|
        \chi_0^j(\tau_{t,0\,}\om, x)
        \,-\,
        \big(\chi_0^j(\tau_{t,0\,} \om, \cdot) \big)_{B(n)}
      \Big|\,
    \md t.
  \end{align*}
  Hence, an application of Lemma~\ref{lemma:sublinearity:l1:chi_0} yields that $\lim_{n \to \infty} I_1(n) = 0$ for $\prob$-a.e.\ $\om$.  Recall that the cocycle property implies that $\Phi(\om, s, x+y) - \Phi(\om, s, x) = \Phi_0(\tau_{s,x\,}\om, y)$ for all $x \in \bbZ^d$ and $|y|=1$.  Hence, a summation by parts (cf.\ \eqref{eq:def:dform} below) gives
  \begin{align*}
    \big(\partial_s \chi(\om, s, \cdot) \big)_{f_n, B(n)}
    &\;=\;
    \big(\partial_s \Phi(\om, s, \cdot) \big)_{f_n, B(n)}
    \overset{\!\eqref{eq:harm_coord}\!}{\;=\;}
    \big( (- \cL_s^{\om} \Phi)(\om, s, \cdot) \big)_{f_n, B(n)}
    \\[1ex]
    &\;=\;
    \frac{c_n}{2 |B(n)|} \sum_{\substack{x \in \bbZ^d\\y \sim 0}} \mspace{-2mu}
    \big(f_n(x+y) - f_n(x)\big)\,
    \vp_y(\tau_{s,x\,} \om),
  \end{align*}
  where we write $\vp_y(\om) \ldef \om_0(0,y) \Phi_0(\om, y)$ for abbreviation.  This yields
  \begin{align*}
    I_2(n)
    &\;\leq\;
    \frac{2 c_n n^d}{|B(n)|}\,
    \sum_{y \sim 0}\,
    \int_0^1
      \bigg|
        \frac{1}{n^{d+1}}
        \int_0^{t n^2}
          \mspace{-6mu}
          \sum_{x \in \bbZ^d}
          \big(f_n(x+y) - f_n(x)\big)\, \vp_y(\tau_{s,x\,} \om)\,
        \md s
      \bigg|\,
    \md t.
  \end{align*}
  Since
  \begin{align*}
    \mean[|\vp_y(\om)|]
    \;=\;
    \mean[\om_0(0,y)\, |\Phi_0(\om, y)|]
    \;\leq\;
    \mean[\om_0(0,y)]^{1/2}\, \Norm{\Phi_0}{L^2_{\mathrm{cov}}}
    \;<\;
    \infty,
  \end{align*}
  $\vp_y \in L^1(\Om, \prob)$.  Thus, a Taylor expansion of $f_n(x+y) - f_n(x)$ combined with \eqref{eq:sublinearity:l1:step1} implies that $\limsup_{n \to \infty} I_2(n) = 0$ $\prob$-a.s, which completes the proof of \eqref{eq:sublinearity:chi:l1:averaged}.
\end{proof}
\begin{proof}[Proof of Proposition~\ref{prop:sublinearity:l1}]
We follow the argument in proof of \cite[Lemma~2]{BFO16} (cf.\ also \cite[Proposition~2.9]{ADS15}). For any $\de > 0$ Lemma~\ref{lemma:sublinearity:chi:l1:averaged} implies that  there exists $n_0 \equiv n_0(\om, \de)$ which is $\prob$-a.s.\ finite such that for all $n \geq n_0$ and $\prob$-a.e. $\om$,
  \begin{align*}
    \frac{1}{n^3}\,
    \int_0^{n^2}
      \frac{1}{|B(n)|} \sum_{x \in B(n)}
      \Big|
        \chi^j(\om, t, x)
        \,-\,
        \big(\chi^j(\om, \cdot, \cdot) \big)_{Q(n)}
      \Big|\,
    \md t
    \;\leq\;
    \de.
  \end{align*}
  Set $c \ldef \max_{n \in \bbN} |B(2n)|/B(n)$ and define $n_k \ldef 2^k n_0$ for any $\bbN \ni k \geq 1$.  Then, by the triangular inequality we find that for $\prob$-a.e. $\om$,
  \begin{align*}
    \frac{1}{n_{k}}\,
    \Big|
      \big(\chi^j(\om, \cdot, \cdot)\big)_{Q(n_k)}
      \,-\,
      \big(\chi^j(\om, \cdot, \cdot)\big)_{Q(n_{k-1})}
    \Big|
    \;\leq\;
    4 \de\, \frac{|B(n_k)|}{|B(n_{k-1})|}
    \;\leq\;
    4 c\, \de.
  \end{align*}
  In particular,
  \begin{align*}
    &\frac{1}{n_k}\,
    \Big|
      \big(\chi^j(\om, \cdot, \cdot)\big)_{Q(n_k)}
      \,-\,
      \big(\chi^j(\om, \cdot, \cdot)\big)_{Q(n_{0})}
    \Big|
    \nonumber\\[.5ex]
    &\mspace{36mu}\leq\;
    \frac{1}{n_k} \sum_{j=1}^k\,
    \Big|
      \big(\chi^j(\om, \cdot, \cdot)\big)_{Q(n_j)}
      \,-\,
      \big(\chi^j(\om, \cdot, \cdot)\big)_{Q(n_{j-1})}
    \Big|
    \;\leq\;
    4 c\, \de\, \sum_{j=1}^{k} \frac{1}{2^{k-j}}
    \;\leq\;
    8 c\, \de.
  \end{align*}
  Thus, for every $k \geq 1$ we obtain that
  \begin{align*}
    I(n_k)\,
    &\ldef\;
    \frac{1}{n_k^3}\,
    \int_0^{n_k^2}\mspace{-10mu}
      \frac{1}{|B(n_k)|} \sum_{x \in B(n_k)}
      \Big| \chi^j(\om, t, x) \Big|\,
    \md t
    \\[.5ex]
    &\;\leq\;
    \de \,+\,
    \frac{1}{n_k}\,
    \Big|
      \big(\chi^j(\om, \cdot, \cdot)\big)_{Q(n_k)}
      \,-\,
      \big(\chi^j(\om, \cdot, \cdot)\big)_{Q(n_{0})}
    \Big|
    \,+\,
    \frac{1}{n_k}\,
    \Big| \big(\chi^j(\om, \cdot, \cdot)\big)_{Q(n_0)} \Big|
    \\[.5ex]
    &\;\leq\;
    (1 + 8c)\, \de
    \,+\,
    \frac{1}{n_k}\,
    \Big| \big(\chi^j(\om, \cdot, \cdot)\big)_{Q(n_0)} \Big|.
  \end{align*}
  Hence, we conclude that for $\prob$-a.e. $\om$
  \begin{align*}
    \limsup_{\de \downarrow 0}\, \limsup_{k \to \infty}
    \frac{1}{n_k^3}\,
    \int_0^{n_k^2}\mspace{-10mu}
      \frac{1}{|B(n_k)|} \sum_{x \in B(n_k)}
      \Big| \chi^j(\om, t, x) \Big|\,
    \md t
    \;\leq\;
    0.
  \end{align*}
  Since $I(n) \leq 4 c\, I(n_k)$ for every $n \geq n_0$ such that $n_{k-1} < n \leq n_k$, the assertion follows.
\end{proof}

\section{Quenched invariance principle} \label{sec:QFCLT}
Throughout this section, which is devoted to the proof of our  main result in Theorem~\ref{thm:main}, we suppose that Assumption~\ref{ass:P} holds. We start with some comments on the construction of the VSRW $X$ and its stochastic completeness as they are not totally obvious in the present time-dependent degenerate situation.

We follow the construction of time-inhomogeneous Markov processes in \cite{St05}.  Let $\{E_{k} : k\geq 1 \}$ be a sequence of independent $\mathop{\mathrm{Exp}}(1)$-distributed random variables.  In order to construct the random walk $X$ under the law $\Prob^{\om}_{s,x}$ we specify its jump times $s < J_1 < J_2 < \ldots $ inductively.  Set $J_0 = s$ and $X_s = x$ and suppose that for any $k \geq 1$ the process $X$ is constructed on $[s, J_{k-1}]$.  Then, $J_k$ is given by
\begin{align*}
  J_k
  \;=\; J_{k-1} +
  \inf\Big\{%
    t \geq 0 \,:\,
    \int_{J_{k-1}}^{J_{k-1}+t}\! \mu_s^\om(X_{J_{k-1}}) \, \md s \geq E_k
  \Big\},
\end{align*}
and at the jump time $t = J_k$ the random walk $X$ jumps according to the transition probabilities $\{\om_t(X_{J_{k-1}},y)/\mu^\om_t(X_{J_{k-1}}), \, y\sim X_{J_{k-1}}\}$. Note that by Assumption~\ref{ass:P}(i) for every $e\in E_d$ the mapping $s\mapsto \om_s(e)$ is $\prob$-a.s.\ locally integrable.
\begin{lemma}
  For $\prob$-a.e.\ $\om$, $\Prob^{\om}_{0,0}$-a.s.\ the process  $(X_t : t\geq 0)$ does not explode, that is there are only finitely many jumps in finite time.
\end{lemma}
\begin{proof}
  We will follow the approach in \cite[Section 5]{DGR15} and consider first a slowed-down process.  Let $\big((T_t, Y_t) : t \geq 0 \big)$ be the Markov process on  $\bbR \times\bbZ^d$ with generator $\cL_Y^\om$ acting on functions $u: \bbR\times \bbZ^d \rightarrow \bbR$ defined by
  \begin{align*}
    \cL_Y^{\om} u(t, x)
    \;=\;
    \frac{1}{1\vee \mu^\om_t(x)}\,
    \big(\partial_t u(t,x) + (\cL_t^ \om u(t,\cdot))(x) \big)
  \end{align*}
  with $\mu_t^\om(x) = \sum_{y\sim x} \om_t(x,y)$.  At point $(t,x)$ the slowed-down process $(Y_t : t \geq 0)$ will jump to $y \sim x$ with rate $\om_{T_t}(x,y)/(1\vee \mu^\om_{T_t}(x))$ and at time $t$ the time process $(T_t : t \geq 0)$ will increase at rate $(1 \vee \mu^\om_t(x))^{-1}$, more precisely
  \begin{align*}
    T_t \;=\; \int_0^t \frac{1}{1 \vee \mu^\om_{T_s}(Y_s)} \,\md s.
  \end{align*}
  Further, notice that the process $X$ can be obtained from $Y$ by a time change, namely
  \begin{align}\label{eq:time_change}
    (X_{t}) \;\overset{d}{=}\; (Y_{T^{-1}_t}),
  \end{align}
 where $T^{-1}$ denotes the right-continuous inverse of $T$. This will allow us to infer non-explosion of the process $X$ from that of $Y$. Clearly, the process $\big((T_t, Y_t) :t \geq 0\big)$ is non-explosive since $T_t \leq t$ and the jump-rates of $Y$ bounded from above by one.

  On the other hand, under Assumption \ref{ass:P} using the irreducibility of the process $Y$ it can be easily seen that the measure
  \begin{align*}
    \frac{1 \vee \mu^\om_0(0)}{\mean[1 \vee \mu^\om_0(0)]}\, d\prob
  \end{align*}
  is stationary and ergodic for the environment process $\big(\tau_{T_t,Y_t}\om : t \geq 0\big)$ (cf.\ e.g.\ \cite[Proposition~2.1]{An14}).  Thus, we may apply the ergodic theorem to obtain that
  \begin{align*}
    \lim_{t \to \infty} \frac{T_t}{t}
    \;=\;
    \frac{1}{\mean[1\vee \mu^\om_0(0)]},
    \qquad (\prob \otimes P_{0,0}^{\om})\text{-a.s.}
  \end{align*}
  In particular, $ \lim_{t \to \infty} T^{-1}_t /t =\mean[1\vee \mu^\om_0(0)]$  and by \eqref{eq:time_change} the process $(X_t : t \geq 0)$ is non-explosive for $\prob$-almost all $\omega$, $P^\om_{0,0}$-almost surely.
\end{proof}
For our purposes the main reason to construct the harmonic coordinates in Section~\ref{sec:harm_coord} is that they allow to decompose the random walk $X$ into a martingale part and a corrector part. We now state this decomposition as a Corollary.
\begin{corro}
  Set $M_t \ldef \Phi(\om, t, X_t)$.  Then, for $\prob$-a.e.\ $\om$, the process $(M_t : t \geq 0)$
  % %
  % \begin{align*}
  %   M_t \;\ldef\; \Phi(\om,t,X_t), \qquad t\geq 0,
  % \end{align*}
  % %
  is a $\Prob^{\om}_{0,0}$-martingale and
  \begin{align} \label{eq:decompX}
    X_t \;=\; M_t + \chi(\om,t,X_t), \qquad t\geq 0.
  \end{align}
  Moreover, for every $v \in \bbR^d$, $v\cdot M$ is a $\Prob^{\om}_{0,0}$-martingale and its quadratic variation process is given by
  \begin{align} \label{eq:vM_qv}
    \langle v \cdot M \rangle_t
    \;=\;
    \int_0^t
      \sum_{y \in \bbZ^d} \om_s(X_s, y) \,
      \left(v \cdot  \big( \Phi_0(\tau_{s,0}\om,y) - \Phi_0(\tau_{s,0}\om, X_s)\big)\right)^2
    \md s.
  \end{align}
\end{corro}
\begin{proof}
  From \eqref{eq:harm_coord} it is immediate that $M$ and hence also $v \cdot M$ are $\Prob^\om_{0,0}$-martingales, in particular their typical paths are c\`{a}dl\`{a}g.  The decomposition in \eqref{eq:decompX} follows directly from the definition of $\chi$.  It remains to show \eqref{eq:vM_qv}.  First note that the op\'erateur carr\'e du champ associated  with $\partial_t+ \mathcal{L}_t^\om$ is given by
  \begin{align*}
    \big(\partial_t + \mathcal{L}_t^\om\big) f^2
    \,-\,
    2f \big(\partial_t + \mathcal{L}_t^\om\big)f
    &\;=\;
    \big(\partial_t (f^2) - 2 f \partial_t f \big)
    \,+\,
    \big( \cL_t^\om (f^2) - 2 f \cL_t^\om f \big)
    \\[.5ex]
    &\;=\;
    \cL_t^\om (f^2) \,-\, 2 f \cL_t^\om f
  \end{align*}
  and
  \begin{align*}
    \big(\cL_t^\om f^2 - 2 f \cL_t^\om f \big) (t,x)
    \;=\;
    \sum_{y\in \bbZ^d} \om_t(x,y) \big( f(t,y)-f(t,x) \big)^2.
  \end{align*}
  Hence,
  \begin{align*}
    \langle v \cdot M \rangle_t
    \;=\;
    \int_0^t
      \sum_{y \in \bbZ^d} \om_s(X_s, y) \,
      \left(v \cdot  \big( \Phi(\om,s,y) - \Phi(\om,s,X_s)\big)\right)^2
    \md s
  \end{align*}
  and \eqref{eq:vM_qv} follows by \eqref{eq:harm_coord2}.
\end{proof}
\begin{lemma}\label{lem:eta}
  The measure $\prob$ is stationary, reversible and ergodic for the environment process $(\tau_{t,X_t}\om : t \geq 0)$.
\end{lemma}
\begin{proof}
  This follows from the ergodicity of the environment and the irreducibility of the process. See \cite[Lemma 2.4]{ADS15} and \cite[Proposition 2.1]{An14} for detailed proofs.
\end{proof}
\begin{prop} \label{prop:mconv}
  Let $M^{(n)}_t \ldef \frac{1}{n} M_{n^2t}$, $t \geq 0$.  Then, for $\prob$-a.e.\ $\om$, the sequence of processes $\{M^{(n)} : n \in \bbN\}$ converges in law in the Skorohod topology to a Brownian motion with a non-degenerate covariance matrix $\Si^2$ given by
  \begin{align*}
  \Si_{ij}^2
  \;=\;
  \mean\!%
  \Big[
    {\textstyle \sum_{x \in \bbZ^d}}\; \om_0(0,x)\,
    \Phi_0^i(\om,x)\, \Phi_0^j(\om,x)
  \Big].
  \end{align*}
\end{prop}
\begin{proof}
  The proof is based on the martingale convergence theorem by Helland (see Theorem 5.1a) in \cite{He82}); the proofs in \cite{ABDH12} or \cite{MP07} can be easily transferred into the time dynamic setting. The argument is based on the fact that the quadratic variation of $M^{(n)}$ converges by an application of the ergodic theorem, since it can be written in terms of the environment process (cf.\ \eqref{eq:vM_qv}), which is ergodic by Lemma~\ref{lem:eta}.

   In order to show that $\Si^2$ is nondegenerate we follow the argument in \cite[Proposition~2.5]{DNS16}. Assume that $v\cdot \Sigma^2 v=0$ for some $v\in \bbR^d$ with $|v|=1$. Then, since $\Phi_0$ satisfies the cocycle-property, we deduce from Lemma~\ref{lemma:L2cov}~(ii) that for $\prob$-a.e.\ $\om$, $v\cdot \Phi_0(\om,x)=0$ for all $x\in \bbZ^d$. Further, using the time-homogeneity of $\Phi_0$ and its continuity w.r.t.\ time we get for $\prob$-a.e.\ $\om$ that $v\cdot \Phi_0(\tau_{t,0}\om,x)=0$ for all $x\in \bbZ^d$ and $t\geq 0$. In view of \eqref{eq:harm_coord2} this implies $v\cdot \Phi(\om,t,x)=0$ for all $x\in \bbZ^d$ and $t\geq 0$. Recall that $x = \chi(\om, t, x) + \Phi(\om,t, x)$. Thus, for $\prob$-a.e.\ $\om$, $|v \cdot x| = |v \cdot \chi(\om,t, x)|$ for all $x \in \bbZ^d$ and $t\geq 0$. In particular,
  \begin{align}\label{eq:nondeg:1}
    \frac{1}{|B(n)|}
    \sum_{x \in B(n)}\mspace{-6mu}
    \big|v \cdot \tfrac{1}{n} x\big|
    \;=\;
    \frac 1 {n^3} \int_0^{n^2}
    \frac{1}{|B(n)|}
    \sum_{x \in B(n)} \mspace{-6mu}
    \big|v \cdot  \chi(\om,t, x)\big| \, \md t.
  \end{align}
  By Proposition~\ref{prop:sublinearity:l1}, the right-hand side of \eqref{eq:nondeg:1} vanishes for $\prob$-a.e.\ $\om$ as $n$ tends to infinity.  On the other hand, for any $\de \in (0, 1)$ we have that
  \begin{align*}
    \frac{1}{n^{d}}\! \sum_{x \in B(n)} \mspace{-6mu}
    \big|v \cdot \tfrac{1}{n} x\big|
    &\;\geq\;
    \frac{\de^2}{n^d}\! \sum_{\substack{x \in B(n)\\x \ne 0}} \mspace{-6mu}
    \indicator_{\{|x| > \de n\}}\; \indicator_{\{|v \cdot x/|x|| > \de\}}
    \\
    &\;\geq\;
    \frac{\de^2}{n^d}
    \bigg(
      |B(n)| \,-\, |B(\de n)|
      \,- \sum_{\substack{x \in B(n)\\ x \ne 0}} \mspace{-4mu}
      \indicator_{\{|v \cdot x/|x|| \leq \de\}}
    \bigg).
  \end{align*}
 Since  $|B(n)| \geq c n^d$ and the other two terms in the bracket above are of order $\de n^d$, by choosing $\de$ sufficiently small, there exists $c > 0$ such that
  \begin{align*}
    \liminf_{n \to \infty}\,
    \frac{1}{|B(n)|}
    \sum_{x \in B(n)}\mspace{-6mu}
    \big|v \cdot \tfrac{1}{n} x\big|
    \;\geq\;
    c
    \;>\;
    0,
  \end{align*}
   which gives a contradiction.  Thus,  $v \cdot \Si^2 v > 0$ for all $v \in \bbR^d\setminus\{0\}$.
\end{proof}
In order to conclude the proof of the invariance principle, an almost sure uniform control of the corrector is required, which is a direct consequence from the sublinearity of corrector established in Proposition~\ref{prop:sublin_corr}.
\begin{prop} \label{prop:contr_corr}
  Suppose that Assumption~\ref{ass:moment} holds and let $T > 0$ be arbitrary. Then, for $\prob$-a.e.\ $\om$,
  \begin{align}
    \sup_{0 \,\leq\, t \,\leq\, T}\,
    \frac{1}{n}\, \Big| \chi\big(\om,n^2t, n\, X_{t}^{(n)}\big) \Big|
    \;\underset{n \to \infty}{\longrightarrow}\;
    0
    \quad \text{ in $\Prob_{\!0,0}^\om$-probability}.
  \end{align}
\end{prop}
\begin{proof}
  Given Proposition~\ref{prop:sublin_corr} this follows by similar arguments as in \cite[Proposition~2.13]{ADS15}, \cite[pp.\ 1884--1885]{FK97} or \cite[p.\ 761]{FK99}.
\end{proof}
Theorem \ref{thm:main} now follows from Proposition \ref{prop:mconv} and Proposition \ref{prop:contr_corr}.

\section{Mean value inequality for time-inhomogeneous Poisson equation} \label{sec:mos_it}
\subsection{Setup and preliminaries}
Let $G = (V, E)$ be an infinite, connected, locally finite graph with vertex set $V$ and (non-oriented) edge set $E$.  We will write $x \sim y$ if $\{x,y\} \in E$.  Moreover, for $A \subset V$ and $x, y \in V$, we will simply write $x \vee y \in A$ for $(x \in A) \vee (y \in A)$.  The graph $G$ is endowed with the counting measure that assigns to any $A \subset V$ simply the number $|A|$ of elements in $A$.  Further, we denote by $B(x,r)$ the closed ball with center $x$ and radius $r$ with respect to the natural graph distance $d$, that is $B(x,r) \ldef \{y \in V \mid d(x,y) \leq \lfloor r \rfloor\}$. Finally, for a set $A\subset V$ we define its boundary by $\partial A \ldef \{ x\in A \,:\, \exists\, y \in V \setminus A \text{ s.\ th. } \{x,y\} \in E \}$.
\smallskip

Throughout this section we will make the following assumption on $G$.
\begin{assumption}\label{ass:graph}
  The graph $G$ satisfies the following conditions:
  \begin{itemize}
  \item[(i)] volume regularity of order $d$ for large balls, that is there exists $d \geq 2$ and $C_{\mathrm{reg}} \in (0, \infty)$ such that for all $x \in V$ there exists $N_1(x) < \infty$ with
    \begin{align}\label{eq:ass:vd}
      C^{-1}_{\mathrm{reg}}\, n^d  \;\leq\; |B(x,n)| \;\leq\; C_{\mathrm{reg}}\, n^d,
      \qquad \forall\, n \geq N_1(x).
    \end{align}
  \item[(ii)] local Sobolev inequality $(S_{d'}^1)$ for large balls, that is there exists $d' \geq d$ and $C_{\mathrm{S_1}} \in (0, \infty)$ such that for all $x \in V$ the following holds. There exists $N_2(x) < \infty$ such that for all  $n \geq N_2(x)$,
    \begin{align}\label{eq:sob:S1}
      \Bigg(\sum_{y \in B(x,n)}\! |u(y)|^{\frac{d'}{d'-1}}\Bigg)^{\!\!\frac{d'-1}{d'}}
      \;\leq\;
      C_{\mathrm{S_1}}\, n^{1 - \frac{d}{d'}}\mspace{-6mu}
      \sum_{\substack{z \vee z' \in B(x,n)\\ \{z,z'\} \in E}}\mspace{-6mu}
      \big|u(z) - u(z') \big|
    \end{align}
    for all $u\! : V \to \bbR$ with $\supp u \subset B(x,n)$.
  \end{itemize}
\end{assumption}

\begin{remark}
 The Euclidean lattice, $(\bbZ^d, E_d)$, satisfies Assumption~\ref{ass:graph} with $d' = d$ and $N_1(x) = N_2(x) = 1$, where (ii) follows from the isoperimetric inequality, see e.g.\ \cite[Theorem~3.2.7]{Ku14}.  For random graphs, e.g.\ supercritical Bernoulli percolation clusters, such an inequality is only satisfied for large sets. There exists $\th \in (0, 1)$ and $N(x) < \infty$, $\prob$-a.s., such that for all $n \geq N(x)$,
  \begin{align*}
    |\partial A| \geq C_{\mathrm{iso}} |A|^{(d-1)/d}
  \end{align*}
  for all connected $A \subset B(x, n)$ with $|A| \geq n^{\th}$, see \cite{BM03, MR04}.  As it was pointed out by M.~Barlow, in such a case Assumption~\ref{ass:graph} (ii) holds with $d' = d/(1-\th)$, see \cite{DNS16}.
\end{remark}
For functions $f\!:A \to \bbR$, where either $A \subseteq V$ or $A \subseteq E$, the $\ell^p$-norm $\norm{f}{p}{A}$ will be taken with respect to the counting measure.  The corresponding scalar products in $\ell^2(V)$ and $\ell^2(E)$ are denoted by $\scpr{\cdot}{\cdot}{V}$ and $\scpr{\cdot}{\cdot}{E}$, respectively.  For any non-empty, finite $B \subset V$ and $p \in (0, \infty)$, we introduce space-averaged norms on functions $f\!: B \to \bbR$ by
\begin{align*}
  \Norm{f}{p, B}
  \;\ldef\;
  \bigg(\frac{1}{|B|}\; \sum_{x \in B}\, |f(x)|^p \bigg)^{\!\!1/p}.
\end{align*}
Moreover, for any non-empty compact interval $I \subset \bbR$ and any finite $B \subset \bbZ^d$ and $p, p' > 0$, we define space-time-averaged norms on functions $u\!: I \times B \to \bbR$ by
\begin{align*}
  \Norm{u}{p, p', I \times B}
  \;\ldef\;
  \bigg(
    \frac{1}{|I|}\; \int_I \Norm{u_t}{p, B}^{p'}\; \md t
  \bigg)^{\!\!1/p'}
  \qquad \text{and} \qquad
  \Norm{u}{p, \infty, I \times B}
  \;\ldef\;
  \max_{t \in I} \Norm{u_t}{p, B},
\end{align*}
where $u_t( \cdot ) \ldef u(t,  \cdot )$ for any $t \in I$.
\begin{lemma}\label{lemma:mos:interp}
  Suppose that $\rho > 1$ and $q' \in [1, \infty]$ are given and $Q \subset \bbR \times V$.  Then, for every $1 < \ga_1 \leq \rho$ and $q' / (q'+1) \leq \ga_2 < \infty$ such that
  \begin{align}\label{eq:mos:interp:cond}
    \frac{1}{\ga_1} \,+\,
    \frac{1}{\ga_2}\, \bigg(1 - \frac{1}{\rho}\bigg) \frac{q'}{q'+1}
    \;=\;
    1
  \end{align}
  the following estimate holds
  \begin{align}\label{eq:mos:interp}
    \Norm{u}{\ga_1, \ga_2, Q}
    \;\leq\;
    \Norm{u}{1, \infty, Q} \,+\, \Norm{u}{\rho, q'/(q'+1), Q}.
  \end{align}
\end{lemma}
\begin{proof}
  This follows by an application of H\"older's and Young's inequality, as in \cite[Lemma~1.1]{KK77}
\end{proof}
Let us endow the graph $G$ with positive, time-dependent weights, that is we consider a family $\om = \{ \om_t(e) :  t \in \bbR, e\in E\} \subset (0, \infty)^{\bbR \times E}$.  Further, we define for any $t \in \bbR$ measures $\mu_t^{\om}$ and $\nu_t^{\om}$ on $V$ by
\begin{align}
  \mu_t^{\om}(x) \;\ldef\; 1 \vee \sum_{x \sim y}\, \om_t(x,y)
  \qquad \text{and} \qquad
  \nu_t^{\om}(x) \;\ldef\; 1 \vee \sum_{x \sim y}\, \frac{1}{\om_t(x,y)}.
\end{align}
It is convenient to introduce a potential theoretic setup.  First, for $f\!: V \to \bbR$ and $F\!: E \to \bbR$ we define the operators $\nabla f\!: E \to \bbR$ and $\nabla^*F\!: V \to \bbR$ by
\begin{align*}
  \nabla f(e) \;\ldef\; f(e^+) - f(e^-),
  \qquad \text{and} \qquad
  \nabla^*F (x)
  \;\ldef\;
  \sum_{e: e^+ =\,x}\! F(e) \,-\! \sum_{e:e^-=\, x}\! F(e),
\end{align*}
where for each non-oriented edge $e \in E$ we specify one of its two endpoints as its initial vertex $e^+$ and the other one as its terminal vertex $e^-$.  Nothing of what will follow depends on the particular choice.  Since  $\scpr{\nabla f}{F}{E} = \scpr{f}{\nabla^* F}{V}$ for all $f \in \ell^2(V)$ and $F \in \ell^2(E)$, $\nabla^*$ can be seen as the adjoint of $\nabla$.  Notice that in the discrete setting the product rule reads
\begin{align}\label{eq:rule:prod}
  \nabla(f g)
  \;=\;
  \av{f} \nabla g \,+\, \av{g} \nabla f,
\end{align}
where $\av{f}(e) \ldef \frac{1}{2}(f(e^+) + f(e^-))$.  Moreover, we denote by $\cL_t^{\om}$ the following time-dependent operator acting on bounded functions $f\!: V \to \bbR$ as
\begin{align*}
  \big(\cL^{\om}_t f\big)(x)
  \;\ldef\;
  \sum_{x \sim y}\, \om_t(x,y)\, \big(f(y) - f(x)\big)
  \;=\;
  - \nabla^*(\om_t \nabla f) (x).
\end{align*}
For any $t \in \bbR$, the \emph{time-dependent Dirichlet form} associated to $\cL_t^{\om}$ is given by
\begin{align} \label{eq:def:dform}
  \cE_t^{\om}(f,g)
  \;\ldef\;
  \scpr{f}{-\cL_t^{\om} g}{V}
  \;=\;
  \scpr{\nabla f}{\om_t \nabla g}{E},
%  \;=\;
%  \scpr{1}{\md \Ga_t^{\om}(f,g)}{E},
\end{align}
and we set $\cE_t^{\om}(f) \ldef \cE_t^{\om}(f,f)$.

Note that \eqref{eq:sob:S1} is a Sobolev inequality on an unweighted graph, while for our purposes we need a version involving the time-dependent weights.
\begin{prop}[local space-time Sobolev inequality]\label{prop:sobolev}
  Consider a graph $(V, E)$ that satisfies Assumption~\ref{ass:graph} with $d' \geq d \geq 2$ and set
  \begin{align}\label{eq:def:rho}
    \rho
    \;\equiv\;
    \rho(d', q)
    \;=\;
    \frac{d'}{d' - 2 + d'/q}.
  \end{align}
  Let $I \subset \bbR$ be a compact interval.  Then, for any $q \in [1, \infty), q' \in [1, \infty]$, there exists $C_{\mathrm{\,S}} \equiv C_{\mathrm{\,S}}(d'\!, q) < \infty$ such that for any $x \in V$ and $n \geq N_1(x) \vee N_2(x)$,
  \begin{align}\label{eq:sob:ineq:2}
    \Norm{u^2}{\rho, q'/(q'+ 1), I \times B(x,n)}
    \;\leq\;
    C_{\mathrm{\,S}}\, n^2\, \Norm{\nu^{\om}}{q, q'\!, I \times B(x,n)}\;
    \bigg(
      \frac{1}{|I|}\, \int_I\, \frac{\cE_{t}^{\om}(u_t)}{|B(x,n)|}\; \md t
    \bigg)
  \end{align}
  for every $u:\bbR \times V \to \bbR$ with $\supp u \subset I \times B(x, n)$.  If $d' > 2$, \eqref{eq:sob:ineq:2} holds for $q = \infty$.
\end{prop}
\begin{proof}
  First, notice that for any $x \in V$ and $n \geq N_1(x) \vee N_2(x)$, \eqref{eq:sob:S1} can be rewritten in the following way
  \begin{align*}
    \Norm{u_t}{\frac{d'}{d'-1}, B(x,n)}
    \;\leq\;
    C_{\mathrm{S_1}}\, n^{1 - \frac{d}{d'}}\, |B(x,n)|^{\frac{1}{d'}}\,
    \frac{\norm{\nabla u_t}{1}{E}}{|B(x,n)|}
    \overset{\eqref{eq:ass:vd}}{\;\leq\;}
    C_{\mathrm{S_1}} C_{\mathrm{reg}}^{1/d'}\, n\,
    \frac{\norm{\nabla u_t}{1}{E}}{|B(x,n)|}
  \end{align*}
  for every $u:\bbR \times V \to \bbR$ with $\supp u_t \subset B(x, n)$ for all $t \in I$.
  Proceeding as in the proof of \cite[Proposition~3.5]{ADS15}, we deduce that there exists $C_{\mathrm{\,S}} < \infty$ such that
  \begin{align*}
    \Norm{u_t^2}{\rho, B(x,n)}
    \;\leq\;
    C_{\mathrm{\,S}}\, n^2\, \Norm{\nu_t^{\om}}{q, B(x,n)}\,
    \frac{\cE_{t}^{\om}(u_t)}{|B(x,n)|}.
  \end{align*}
  Thus, for any $q' \geq 1$ the assertion follows by H\"older's inequality.
\end{proof}

\subsection{Maximal inequality via Moser iteration}
In this section, our main objective is to establish a maximum inequality for the solution of a particular Poisson equation having a right-hand side in divergence form.  More precisely, we denote by $u:\bbR\times V \rightarrow \bbR$ a solution of
\begin{align}\label{eq:poisson_eq}
  \partial_t u + \cL_t^{\om} u \;=\; \nabla^* V_t^{\om}   ,
  \qquad \text{on} \quad Q=I \times B,
\end{align}
where $I = [s_1, s_2] \subset \bbR$ is an interval, $B \subset V$ is a finite, connected subset of $V$ and $V_t^{\om}\!: \bbR \times E \to \bbR$ is given by
\begin{align}\label{eq:local:drift}
  V_t^{\om}(e)
  \;\ldef\;
  \om_t(e)\, \nabla f(e)
\end{align}
for some function $f\!:V \to \bbR$.

For any $x_0 \in V$, $t_0 \geq 0$ and $n \geq 1$, we denote by $Q(n) \equiv [t_0, t_0 + n^2] \times B(x_0, n)$ the corresponding time-space cylinder, and we set
\begin{align*}
  Q(\si n)
  \;=\;
  [t_0, t_0 + \si n^2] \times B(x_0, \si n), \qquad \sigma\in [0,1].
\end{align*}
Now we are ready to state the main result of this section.

\begin{theorem} \label{thm:max_ineq}
  Suppose that Assumption~\ref{ass:graph} holds for some $d' \geq d \geq 2$.  Assume that $u$ solves $\partial_t u + \cL_t^{\om} u = \nabla^* V_t^{\om}$ on $Q(n)$, where the function $f$ in \eqref{eq:local:drift} satisfies $|\nabla f(e)| \leq 1/n$ for all $e \in E$.  Then, for any $\De \in [0, 1)$ and $p, p', q, q' \in (1, \infty]$ satisfying
  \begin{align}\label{eq:cond:pq}
    \frac{1}{p} \cdot \frac{p'}{p'-1} \cdot \frac{q'+1}{q'}
    \,+\,
    \frac{1}{q}
    \;<\;
    \frac{2}{d'}
  \end{align}
  there exist $N(\De) < \infty$, $\ga \equiv \ga(d'\!,p, p'\!, q, q') \in (0,1]$, $\ka \equiv \ka(d'\!, p, p', q, q') \in (1, \infty)$ and $C_1 \equiv C_1(d) < \infty$ such that for all $n \geq \max\{2 N_1(x_0), 2 N_2(x_0), N(\De)\}$ and $1/2 \leq \si' < \si \leq 1$ with $\si - \si' > n^{-\De}$,
  \begin{align}\label{eq:mos:general}
    &\max_{(t,x) \in Q(\si' n)} |u(t,x)|
    \nonumber\\[.5ex]
    &\mspace{36mu}\leq\;
    C_1\,
    \Bigg(
      \frac{\Norm{1 \vee \mu^{\om}}{p, p'\!, Q(n)}\,
        \Norm{1 \vee \nu^{\om}}{q, q'\!,Q(n)}}{(\si - \si')^2}
    \Bigg)^{\!\!\ka}\,
    M_{\gamma}\big(\Norm{u}{2\al p_*, 2\al p'_*, Q(\si n)}\big),
  \end{align}
  with $\al > 1$ defined in \eqref{eq:def:alpha}, $M_{\gamma}(s) \ldef s^{\gamma} \vee s$ and $p_* \ldef p/(p-1)$, $p_*' \ldef p'/(p'-1)$.
\end{theorem}
We now turn to the proof of Theorem~\ref{thm:max_ineq}. As a first step we prove the following energy estimate for solutions of \eqref{eq:poisson_eq}.
\begin{lemma} \label{lem:mos:DF}
  Suppose that $Q = I \times B$, where $I = [s_1, s_2]$ is an interval and $B$ is a finite, connected subset of $V$.  Consider a smooth function $\ze\!: \bbR \to [0, 1]$ with $\ze = 0$ on $[s_2, \infty)$ and a function $\eta\!: V \to [0, 1]$ such that
  \begin{align*}
    \supp \eta \;\subset\; B
    \qquad\text{and} \qquad
    \eta \;\equiv\; 0 \quad \text{on} \quad \partial B.
  \end{align*}
  Further, let $u$ be a solution of \eqref{eq:poisson_eq} on $Q$.  Then, there exists $C_2 < \infty$ such that for all $\al \geq 1$ and $p, p', p_*, p'_* \in (1, \infty)$ with $1/p + 1/p_* = 1$ and $1/p' + 1/p'_* = 1$,
  \begin{align}\label{eq:mos:energy}
    &\frac{1}{|I|}\, \Norm{\ze (\eta\, \tilde{u}^{\al})^2}{1, \infty, Q}
    \,+\,
    \frac{1}{|I|}
    \int_{I} \ze(t)\; \frac{\cE_{t, \eta^2}^{\om}(\tilde u_t^{\al})}{|B|} \, \md t
    \nonumber\\[.5ex]
    &\mspace{72mu}\leq\;
    C_2\, \al^2\, \Norm{\mu^{\om}}{p, p', Q}\,
    \Big(
      \norm{\nabla \eta}{\infty}{E}^2
      +
      \|\ze'\|_{\raisebox{-.0ex}{$\scriptstyle L^{\raisebox{.2ex}{$\scriptscriptstyle {\!\infty\!}$}} (I)$}}
    \Big)\, \Norm{|u|^{2\al}}{p_*, p_*', Q}
    \nonumber\\[.5ex]
    &\mspace{108mu}+\,
    C_2\, \al^2\, \Norm{\mu^{\om}}{p, p', Q}\,
    \norm{(\nabla \eta)(\nabla f)}{\infty}{E}\, \Norm{|u|^{2\al-1}}{p_*, p_*', Q}
    \nonumber\\[1ex]
    &\mspace{108mu}+\,
    C_2\, \al^2\ \Norm{\mu^{\om}}{p, p', Q}\,
    \norm{\nabla f}{\infty}{E}^2\, \Norm{|u|^{2\al-2}}{p_*, p_*', Q}
  \end{align}
  where $\tilde{u}^{\al} \ldef |u|^{\al} \cdot \sign u$, $\cE_{t, \eta^2}^{\om}(g) \ldef \scpr{\nabla g}{\av{\eta^2}\, \om_t\nabla g}{E}$ and $f$ being the function appearing in \eqref{eq:local:drift}.
\end{lemma}
\begin{proof}
  Let us consider a function $u$ such that $\partial_t u + \cL_t^\om\, u = \nabla^* V_t^\om$ on $Q = I \times B$.  Then, for any $t \in I$, a summation by parts yields
  \begin{align}\label{eq:poisson:weak}
    \frac{1}{2 \al}\, \partial_t \norm{\eta\, \tilde{u}_t^{\al}}{2}{V}^2
    \;=\;
    \scpr{\nabla(\eta^2 \tilde{u}_t^{2\al-1})}{\om_t \nabla u_t}{E}
    \,+\,
    \scpr{\nabla(\eta^2 \tilde{u}_t^{2\al-1})}{V_t^\om}{E}.
  \end{align}
  Proceeding as in the proof of \cite[Lemma~3.2]{ADS16a}, we will estimate the terms on the right-hand side of \eqref{eq:poisson:weak} separately.  Let us point out that the constants $c \in (0, \infty)$  appearing in the computations below, is independent of $\al$, but may change from line to line.  In view of \eqref{eq:A1:pol:ub}, we have that
  \begin{align*}
    \scpr{\av{\eta^2} \nabla \tilde{u}_t^{2\al-1}}{\om_t \nabla u_t}{E}
    \;\geq\;
    \frac{2\al-1}{\al^2}\, \cE_{t, \eta^2}^{\om}\big(\tilde{u}_t^{\al}\big)
    \;\geq\;
    \frac{1}{\al}\, \cE_{t, \eta^2}^{\om}\big(\tilde{u}_t^{\al}\big).
  \end{align*}
  On the other hand, by \eqref{eq:A1:chain:ub2} and Young's inequality, that reads $|a b| \leq \frac{1}{2}(\ve a^2 + b^2/\ve)$ for $\ve \in (0, \infty)$, we obtain that
  \begin{align*}
    \scpr{\av{\tilde{u}_t^{2\al-1}} \nabla \eta^2}{\om_t \nabla u_t}{E}
    &\;\geq\;
    -c\,
    \norm{\om_t (\nabla \tilde{u}_t^{\al})(\nabla \eta^2) \av{|u_t|^{\al}}}{1}{E}
    \nonumber\\[.5ex]
    &\;\geq\;
    -c\, \ve\, \cE_{t, \eta^2}^{\om}\big(\tilde{u}_t^{\al}\big)
    \,-\,
    \frac{c}{\ve}\, \norm{\nabla \eta}{\infty}{E}^2\,
    \norm{|u_t|^{2\al} \mu_t^{\om}}{1}{B},
  \end{align*}
  where we used that $\nabla \eta^2 = 2 \av{\eta} (\nabla \eta)$ and $\av{\eta}^2 \leq 2 \av{\eta^2}$.  Hence,
  % by the above estimates the first term on the right-hand side of \eqref{eq:poisson:weak} is bounded from below by
  %
  \begin{align}\label{eq:rhs:poisson1}
    &\scpr{\nabla(\eta^2 \tilde{u}_t^{2\al-1})}{\om_t \nabla u_t}{E}
    \nonumber\\[.5ex]
    &\mspace{36mu}\geq\;
    \bigg(\frac{1}{\al} - c\, \ve \bigg)\,
    \cE_{t, \eta^2}^{\om}\big(\tilde{u}_t^{\al}\big)
    \,-\,
    \frac{c}{\ve}\, \norm{\nabla \eta}{\infty}{E}^2\,
    \norm{|u_t|^{2\al} \mu_t^{\om}}{1}{B}.
  \end{align}
  Next, we consider the second term on the right-hand side of \eqref{eq:poisson:weak}.  Since $\eta \in [0, 1]$,
  \begin{align*}
    \scpr{\av{\tilde{u}_t^{2\al-1}} \nabla \eta^2}{V_t^{\om}}{E}
    \;\geq\;
    -c\, \norm{(\nabla \eta)(\nabla f)}{\infty}{E}\,
    \norm{|u_t|^{2\al-1} \mu_t^{\om}}{1}{B}.
  \end{align*}
  By applying \eqref{eq:A1:chain:ub1} and Young's inequality, we find for any $\al \geq 1$,
  \begin{align*}
    \scpr{\av{\eta^2} \nabla \tilde{u}_t^{2\al-1}}{\om_t \nabla f}{E}
    &\;\geq\;
    -c\,
    \norm{\om_t \av{\eta^2} \av{|u_t|^{\al-1}} (\nabla \tilde{u}_t^{\al})(\nabla f)}{1}{E}
    \\[.5ex]
    &\;\geq\;
    -c\, \ve\, \cE_{t, \eta^2}^{\om}\big(\tilde{u}_t^{\al}\big)
    \,-\,
    \frac{c}{\ve}\, \norm{\nabla f}{\infty}{E}^2\,
    \norm{|u_t|^{2\al-2} \mu_t^{\om}}{1}{B}.
  \end{align*}
  Hence, by combining these estimates, we obtain that the second term on the right-hand side of \eqref{eq:poisson:weak} is bounded from below by
  \begin{align}\label{eq:rhs:poisson2}
    \scpr{\nabla(\eta^2 \tilde{u}_t^{2\al-1})}{V_t^{\om}}{E}
    &\;\geq\;
    -\, c\, \ve\, \cE_{t, \eta^2}^{\om}\big(\tilde{u}_t^{\al}\big)
    \,-\,
    \frac{c}{\ve}\, \norm{\nabla f}{\infty}{E}^2\,
    \norm{|u_t|^{2\al-2} \mu_t^{\om}}{1}{B}
   \nonumber\\[.5ex]
   &\mspace{36mu}-\,
   c\, \norm{(\nabla \eta)(\nabla f)}{\infty}{E}\,
   \norm{|u_t|^{2\al-1} \mu_t^{\om}}{1}{B}.
  \end{align}
  Thus, in view of \eqref{eq:rhs:poisson1} and \eqref{eq:rhs:poisson2} and by choosing $\ve = 1/(4 \,c\, \al)$, we deduce from \eqref{eq:poisson:weak} that there exists $C_2 < \infty$ such that
  \begin{align}\label{eq:poisson:energy}
    -\partial_t \Norm{(\eta\, \tilde{u}_t^{\al})}{2, B}^2
    \,+\,
    \frac{\cE_{t, \eta^2}^{\om}\big(\tilde{u}_t^{\al}\big)}{|B|}
    &\;\leq\;
    \frac{C_2}{2}\, \al^2\, \norm{\nabla \eta}{\infty}{E}^2\,
    \Norm{|u_t|^{2\al} \mu_t^{\om}}{1, B}
    \nonumber\\
    &\mspace{36mu}+\,
    \frac{C_2}{2}\, \al^2\,
    \norm{\nabla f}{\infty}{E}^2\, \Norm{|u_t|^{2\al-2} \mu_t^{\om}}{1, B}
    \nonumber\\[1ex]
    &\mspace{36mu}+\,
    \frac{C_2}{2}\, \al^2\, \norm{(\nabla \eta)(\nabla f)}{\infty}{E}\,
    \Norm{|u_t|^{2\al-1} \mu_t^{\om}}{1, B}.
  \end{align}
  Moreover, since $\ze(s_2) = 0$,
  \begin{align*}
    \int_s^{s_2}\!\!
      \!- \ze(t)\, \partial_t \Norm{(\eta\, \tilde{u}_t^{\al})}{2, B}^2\;
    \md t
    &\;=\;
    \int_s^{s_2}
      \Big(
        -\partial_t\big( \ze(t) \Norm{(\eta\, \tilde{u}_t^{\al})}{2, B}^2 \big)
        +
        \ze'(t) \Norm{(\eta\, \tilde{u}_t^{\al})}{2,B}^2
      \Big)\;
    \md t
    \\[.5ex]
    &\;\geq\;
    \ze(s) \Norm{(\eta\, \tilde{u}_s^{\al})}{2, B}^2
    \,-\,
    \|\ze'\|_{\raisebox{-.0ex}{$\scriptstyle L^{\raisebox{.2ex}{$\scriptscriptstyle {\!\infty\!}$}} (I)$}}
    \, \int_{s_1}^{s_2} \Norm{|u_t|^{2\al}}{1, B}\; \md t
  \end{align*}
  for any $s \in [s_1, s_2)$.  Thus, by multiplying both sides of \eqref{eq:poisson:energy} with $\ze$ and integrating the resulting inequality over $[s, s_2]$ for any $s \in I$ and  by applying the H\"older and Jensen inequality one obtains the inequality \eqref{eq:mos:energy} separately for each of the two terms in the left-hand side of \eqref{eq:mos:energy}.
\end{proof}
\begin{proof}[Proof of Theorem~\ref{thm:max_ineq}]
  For any $p, p' \in (1, \infty)$, let $p_* \ldef p/(p-1)$ and $p_*' \ldef p'/(p'-1)$ be the H\"older conjugate of $p$ and $p'$, respectively.  Further, set
  \begin{align}\label{eq:def:alpha}
    \al
    \;\ldef\;
    \frac{1}{p_*}
    \,+\,
    \frac{1}{p_*'}\bigg(1 - \frac{1}{\rho} \bigg)\, \frac{q'}{q'+1}
    \qquad \text{and} \qquad
    \al_k
    \;\ldef\;
    \al^k,
  \end{align}
  where $\rho$ is defined in \eqref{eq:def:rho}.  Notice that for any $p, p', q, q' \in (1,\infty]$ for which \eqref{eq:cond:pq} is satisfied, $\al > 1$ and therefore $\al_k \geq 1$ for every $k \in \bbN_0$.  In particular, $\al > 1$ implies that $\al p_*' > q'/(q'+1)$ and $\al p_* \leq \rho$ so that Lemma~\ref{lemma:mos:interp} is applicable.

  For some $\De \in [0,1)$, let $n \geq 2(N_1(x_0) \vee N_2(x_0)) \vee N(\De)$, where $N(\De) < \infty$ is such that $n^{1 - \De} / (\ln n)^{(\ln 2)/(\ln \al)} \vee \me^{\al^2} \geq 2$ for all $n \geq N(\De)$.  Set $K \ldef \lfloor (\ln \ln n) / (\ln \al) \rfloor$.  In the sequel, fix some $1/2 \leq \si' < \si \leq 1$ with $\si - \si' > n^{-\De}$, and consider a sequence $\{ Q(\si_k n) : k \in \bbN_0\}$ of space-time cylinders, where
  \begin{align*}
    \si_k \;=\; \si' + 2^{-k} (\si - \si')
    \qquad \text{and} \qquad
    \tau_k \;=\; 2^{-k-1} (\si - \si'),
    \quad k \in \bbN_0.
  \end{align*}
  In particular, we have that $\si_k = \si_{k+1} + \tau_k$ and $\si_0 = \si$.  For abbreviation we write $I_k \ldef [t_0, t_0 + \si_k n^2]$, $B_k \ldef B(x_0, \si_k n)$ and $Q_k \ldef I_k \times B_k$.  Note that $|I_k|/|I_{k+1}| \leq 2$ and $|B_k| / |B_{k+1}| \leq C_{\mathrm{reg}}^2 2^d$.

  Let us emphasise that $B_{k} \subsetneq B_{k-1}$ for any $k = 1, \ldots, K$, since
  \begin{align*}
    (\si_{k-1} - \si_k)n
    \;=\;
    2^{-k}(\si - \si')n
    \;>\;
    \frac{n^{1-\De}}{(\ln n)^{(\ln 2)/(\ln \al)}}
    \;\geq\;
    2,
    \qquad \forall\, k \in \{1, \ldots, K\}.
  \end{align*}
  Hence, we can define a sequence $\{\eta_k : k \in \bbN_0\}$ of cut-off functions in space such that $\supp \eta_k \subset B_k$, $\eta_k \equiv 1$ on $B_{k+1}$, $\eta_k \equiv 0$ on $\partial B_k$ and $\norm{\nabla \eta_k}{\infty}{E} \leq 1/\tau_{k} n$.  Moreover, let $\{\ze_k \in C^{\infty}(\bbR) : k \in \bbN_0\}$ be a sequence of smooth cut-off functions in time having the properties that $\ze_k \equiv 1$ on $I_{k+1}$, $\ze_k \equiv 0$ on $[t_0 + \si_k n^2, \infty)$ and $\| \ze_k' \|_{\raisebox{-0ex}{$\scriptstyle L^{\raisebox{.1ex}{$\scriptscriptstyle \!\infty$}} (\bbR)$}} \leq  1 / \tau_k n^2$.

  First, in view of \eqref{eq:mos:interp} we have that
  \begin{align}\label{eq:mos:split}
    \Norm{\tilde{u}^{2 \al_k}}{\al p_*, \al p_*', Q_{k+1}}
    \;\leq\;
    \Norm{\tilde{u}^{2\al_k}}{1, \infty, Q_{k+1}}
    \,+\,
    \Norm{\tilde{u}^{2\al_k}}{\rho, q'/(q'+1), Q_{k+1}}.
  \end{align}
  By applying the space-time Sobolev inequality \eqref{eq:sob:ineq:2} to $\ze_k\,\eta_k \tilde{u}_t^{\al_k}$ and using that
  \begin{align*}
    \cE_t^{\om}(\eta_k \tilde{u}_t^{\al_k})
    \;\leq\;
    2\, \cE_{t, \eta_k^2}(\tilde{u}_t^{\al_k})
    \,+\,
    2\,\norm{\nabla \eta_k}{\infty}{E}^{2}\,
    \norm{|u_t|^{2\al_k} \mu_t^{\om}}{1\!}{B_k}
  \end{align*}
  we obtain
  \begin{align*}
    &\Norm{\tilde{u}^{2\al_k}}{\rho, q'/(q'+1), Q_{k+1}}
    \\[.5ex]
    &\mspace{36mu}\leq\;
    c\, n^2 \Norm{\nu^{\om}}{q, q', Q_k}\,
    \Bigg(
      \frac{1}{|I_k|}\,
      \int_{I_k}\!
        \ze_k(t)\,
        \frac{\cE_{t,\eta_k^2}^{\om}(\tilde{u}_t^{\al_k})}{|B_k|}\;
      \md t
      \,+\,
      \frac{1}{(\tau_k n)^2}\, \Norm{|u|^{2 \al_k} \mu^{\om}}{1, 1, Q_k}
    \Bigg).
  \end{align*}
  Recall that $|\nabla f(e)| \leq 1/n$ for all $e \in E$.  Thus, by means of Jensen's inequality, the energy estimate \eqref{eq:mos:energy} implies that
  \begin{align}\label{eq:mos:max:time}
    &\frac{1}{|I_k|}\, \Norm{\tilde{u}^{2 \al_k}}{1, \infty, Q_{k+1}}
    %\Norm{\ze_k (\eta_k \tilde{u}^{\al_k})^2}{1, \infty, Q_k}
    \,+\,
    \frac{1}{|I_k|}
    \int_{I_k}\!
      \ze_k(t)\, \frac{\cE_{t, \eta_k^2}^{\om}(\tilde u_t^{\al_k})}{|B_k|}\;
    \md t
    \nonumber\\[.5ex]
    &\mspace{144mu}\leq\;
    c\, \Norm{\mu^{\om}}{p, p', Q_k} \bigg(\frac{\al_k}{\tau_k n} \bigg)^{\!\!2}
    \Norm{u}{2 \al_k p_*, 2 \al_k p_*', Q_k}^{2 \al_k \ga_k},
  \end{align}
  where $\ga_k = 1$ if $\Norm{u}{2 \al_k p_*, 2 \al_k p_*', Q_k} \geq 1$ and $\ga_k = 1 - 1/\al_k$ if $\Norm{u}{2 \al_k p_*, 2 \al_k p_*', Q_k} < 1$.  Hence, by combining these two estimates with \eqref{eq:mos:split}, we find that
  \begin{align}\label{eq:mos:iter:step}
    &\Norm{u}{2 \al_{k+1} p_*, 2\al_{k+1} p_*', Q_{k+1}}
    \nonumber\\[.5ex]
    &\mspace{36mu}\leq\;
    \bigg(
      c\, \frac{2^{2k} \al_k^2}{(\si - \si')^2}\,
      \Norm{1 \vee \mu^{\om}}{p, p', Q(n)}\, \Norm{1 \vee \nu^{\om}}{q, q', Q(n)}
    \bigg)^{\!\!1/(2\al_k)}\,
    \Norm{u}{2 \al_{k} p_*, 2\al_{k} p_*', Q_{k}}^{\ga_k}.
  \end{align}
  Further, observe that $|B_{K}|^{1/2\al_{K-1}} \leq c < \infty$ uniformly in $n$.  Hence, an application of \eqref{eq:mos:max:time} yields
  \begin{align*}
    &\max_{(t,x) \in Q(\si' n)} |u(t,x)|
    \\[.5ex]
    &\mspace{36mu}\leq\;
    \max_{(t,x) \in Q_{K}} |u(t,x)|
    \;\leq\;
    |B_{K}|^{1/(2\al_{K-1})}\,
    \Norm{\tilde{u}^{2\al_{K-1}}}{1, \infty, Q_{K}}^{1/(2\al_{K-1})}
    \nonumber\\[.5ex]
    &\mspace{36mu}\leq\;
    \bigg(
      c\,
      \frac{2^{2 (K-1)} \al_{K-1}^2}{(\si - \si')^2}\, 
      \Norm{\mu^{\om}}{p, p', Q(n)}
    \bigg)^{\!\!1/(2\al_{K-1})}
    \Norm{u}{2\al_{K-1} p_*, 2 \al_{K-1} p_*', Q_{K-1}}^{\ga_{K-1}}.
  \end{align*}
  By iterating the inequality \eqref{eq:mos:iter:step}, we get
  \begin{align*}
    &\max_{(t,x) \in Q(\si' n)} |u(t,x)|
    \\[.5ex]
    &\mspace{36mu}\leq\;
    C_2\, \prod_{k=1}^{K-1}
    \Bigg(
      \frac{\Norm{1 \vee \mu^{\om}}{p, p', Q(n)}\, 
        \Norm{1 \vee \nu^{\om}}{q, q',Q(n)}}{(\si - \si')^2}
    \Bigg)^{\!\!1/(2\al_k)}\!\!
    M_{\gamma}\big(\Norm{u}{2 \al p_*, 2 \al p_*', Q(\si n)}\big),
  \end{align*}
  where $0 < \ga = \prod_{k=1}^{\infty} (1-1/\al_k) \leq 1$ and $C_2 < \infty$ is a constant independent of $k$, since $\sum_{k=0}^{\infty}k/\al_k < \infty$.  Finally, by choosing $\ka = \frac{1}{2} \sum_{k=0}^\infty 1/\al_k <\infty$, the claim follows.
\end{proof}
\begin{corro}
  Suppose that the assumptions of Theorem~\ref{thm:max_ineq} are satisfied.  Additionally, assume that there exist $C_{3}< \infty$ and $N_3(x_0, t_0) < \infty$ such that
  \begin{align}\label{eq:bound:claim}
    \max_{t \in [t_0, t_0 + n^2]} \Norm{u_t}{1, B(x_0, n)}
    \;\leq\;
    C_3
    \qquad \forall\, n \geq N_3(x_0, t_0).
  \end{align}
  Then, there exists $\ga' \equiv \ga'(d', p, p', q, q') \in (0,1]$, $\ka' \equiv \ka'(d',p,p',q,q') \in (1, \infty)$ and $C_4(C_3) < \infty$  such that for all $n \geq \max\{2N_1(x_0), 2N_2(x_0), N_3(x_0, t_0)\}$
  \begin{align}\label{eq:max_ineq:improved}
    &\max_{(t,x) \in Q(\frac{1}{2} n)} |u(t,x)|
    % \nonumber\\[.5ex]
    % &\mspace{36mu}\leq\;
    \;\leq\;    
    C_4\,
    \Big(
      \Norm{1 \vee \mu^{\om}}{p, p'\!, Q(n)}\,
      \Norm{1 \vee \nu^{\om}}{q, q'\!, Q(n)}
    \Big)^{\!\!\ka'}\,
    M_{\ga'}\big(\Norm{u}{1, 1, Q(n)}\big).
  \end{align}
\end{corro}
\begin{proof}
  We shall show that, in view of \eqref{eq:bound:claim}, the assertion of Theorem~\ref{thm:max_ineq} can be strengthened by adapting the arguments in \cite[Theorem 2.2.3]{S-C02}.  In the sequel, let $\si' \ldef 1/2$, $\si \ldef 1$ and set $\si_k \ldef \si - 2^{-k}(\si - \si')$ for any $k \in \bbN_0$.  Note that $\si_k$ increases in $k$ with $\si_0 = \si'$ and $\lim_{k \to \infty} \si_k = \si$.  Then, by H\"older's inequality we have
  \begin{align*}
    \Norm{u}{2\al p_*, 2\al p_*', Q(\si_k n)}
    \;\leq\;
    \Norm{u}{1, 1, Q(\si_k n)}^{\th}\,
    \Norm{u}{\infty, \infty, Q(\si_k n)}^{1-\th}
    \qquad \forall\, k \in \bbN_0,
  \end{align*}
  where $\th = 1 / \max\{2\al p_*, 2\al p_*'\}$.  Set $K \ldef \lfloor (\ln n) / (2 \ln 2)  \rfloor$.  Then, for all $n \geq 2^4$, it holds that $K \geq 2$ and $(\si_k - \si_{k-1}) \geq 1/\sqrt{n} > n^{-3/4}$ for any $k \in \{1, \ldots, K-1\}$.  By applying the maximal inequality \eqref{eq:mos:general}, we get for every $k \in \{1, \ldots, K-1\}$,
  \begin{align*}
    \Norm{u}{\infty, \infty, Q(\si_{k-1} n)}
    \;\leq\;
    2^{2 \ka k}\, J\,
    M_{\ga}\big(
      \Norm{u}{1, 1, Q(\si n)}^{\th}\,
      \Norm{u}{\infty, \infty, Q(\si_k n)}^{1 - \th}
    \big),
  \end{align*}
  where we introduced $J = c\, \big( \Norm{1 \vee \mu^{\om}}{p, p', Q(n)}\, \Norm{1 \vee \nu^{\om}}{q, q',Q(n)} \big)^{\ka}$ to simplify notation.  Further, note that $M_{\ga}(x y) \leq M_{\ga}(x) M_{\ga}(y)$ and $M_{\ga^i}(M_{\ga^j}(x)) = M_{\ga^{i+j}}(x)$ for all $x,y \in [0, \infty)$ and $i,j \in \bbN_0$.  By iterating the inequality above, we obtain
  \begin{align*}
    \Norm{u}{\infty, \infty, Q(\si_0 n)}
    &\;\leq\;
    2^{2\ka \sum_{k=0}^{K-2} (k+1)(1 - \th)^k}
    J^{\sum_{k=0}^{K-2}(1-\th)^k}
    \Bigg(  
      \prod_{k=0}^{K-2}
      M_{\ga^{k+1}}\big(\Norm{u}{1, 1, Q(\si n)\!}^{\th}\big)^{(1-\th)^k}
    \Bigg)
    \\[.5ex]  
    &\mspace{36mu}\times\,
    M_{\ga^{K-1}}\big(
      \Norm{u}{\infty, \infty, Q(\si_{K-1} n)}
    \big)^{(1 - \th)^{K-1}}.
  \end{align*}
  Further, observe that $\prod_{k=0}^{K-2} M_{\ga^{k+1}}\big(\Norm{u}{1, 1, Q(\si n)\!}^{\th}\big)^{(1-\th)^k} \leq M_{\ga \th}\big( \Norm{u}{1, 1, Q(\si n)\!} \big)$ and $|B(x_0, \si_{K-1} n)|^{(1-\th)^{K-1}} \leq c < \infty$ uniformly in $n$.  Hence, in view of \eqref{eq:bound:claim},
  \begin{align*}
    \Norm{u}{\infty, \infty, Q(\si' n)}
    \;\leq\;
    c\, C_3\, 2^{2\ka/\th^2}\, J^{1/\th}\,
    M_{\ga \th}\big(\Norm{u}{1, 1, Q(\si n)\!} \big),
  \end{align*}
  and the assertion \eqref{eq:max_ineq:improved} follows with $\ka' = \ka/\th$ and $\ga' = \ga \th$.
\end{proof}
\begin{proof}[Proof of Proposition~\ref{prop:maxin_lattice}]
  Recall that $\bbZ^d$ satisfies Assumption~\ref{ass:graph} with $d'=d$ and $N_1(x) = N_2(x) = 1$.  Moreover, for any $j \in \{1, \ldots, d\}$, the function $(t, x) \mapsto n^{-1} \chi^j(\om, t, x)$ solves the equation $\partial_t u + \cL_t^{\om} u = \nabla^* V_t^{\om}$ on $\bbR \times \bbZ^d$ with $V_t^{\om}(e) = \om_t(e) \nabla f(e)$ and $f(x) = n^{-1} x^j$.  Then, the assertion of Proposition~\ref{prop:maxin_lattice} follows from \eqref{eq:max_ineq:improved} with the choice $x_0 = 0$, $t_0 = 0$ and $n$ replaced by $2n$ once we have shown that \eqref{eq:bound:claim} holds true.

  In order to prove \eqref{eq:bound:claim}, recall that $x^j = \Phi^j + \chi^j$, where $(t,x) \mapsto \Phi^j(\om, t, x)$ solves the equation $\partial_t u + \cL_t^{\om} u = 0$ on $\bbR \times \bbZ^d$.  Hence, it suffices to show that $\prob$-a.s.
  \begin{align}\label{eq:bound:claim2}
    \limsup_{n \to \infty} \sup_{t \in [0, n^2]} 
    \Norm{n^{-1} \Phi^j(\om, t, \cdot)}{1, B(0, n)}
    \;\leq\;
    c
    \;<\;
    \infty.
  \end{align}
  Set $u = n^{-1} \Phi^j$, and rewrite the solution $u$ as $u = v^+ - v^-$, where
  \begin{align*}
    v_t^+(x)
    \;\ldef\;
    \sqrt{u_t^2(x)+ 1}
    \qquad \text{and} \qquad
    v_t^-(x)
    \;=\;
    \sqrt{u_t^2(x) + 1} - u_t(x)
  \end{align*}
  are positive subsolutions, that is $\partial_t v^{\pm} + \cL_t^{\om} v^{\pm} \geq 0$.  Further, let $\eta\!: \bbZ^d \to [0, 1]$ be a cut-off function in space with the property that $\supp \eta \subset B(2n)$, $\eta \equiv 1$ on $B(n)$, $\eta \equiv 0$ on $\partial B(2n)$ and $\norm{\nabla \eta}{\infty}{E} \leq 1/n$.  Moreover, let $\ze\!: \bbR \to [0,1]$ be a smooth cut-off function in time such that $\ze \equiv 1$ on $[0, n^2]$, $\ze \equiv 0$ on $[2n^2, \infty)$ and $\| \ze' \|_{\raisebox{-0ex}{$\scriptstyle L^{\raisebox{.1ex}{$\scriptscriptstyle \!\infty$}} (\bbR)$}} \leq  c / n^2$.  Then,
  \begin{align*}
    \partial_t\Big(\ze(t) \scpr{\eta}{v_t^{\pm}}{\bbZ^d}\Big)
    &\;\geq\;
    \ze(t)\, \scpr{\nabla \eta}{\om_t \nabla v_t^{\pm}}{E_d} 
    + \ze'(t)\, \scpr{\eta}{v_t^{\pm}}{\bbZ^d}
    \\[.5ex]
    &\;\geq\;
    - \scpr{|\nabla \eta|}{\om_t |\nabla v_t^{\pm}|}{E_d}
    - \| \ze' \|_{\raisebox{-0ex}{$\scriptstyle L^{\raisebox{.1ex}{$\scriptscriptstyle \!\infty$}} (\bbR)$}}
    \scpr{\eta}{v_t^{\pm}}{\bbZ^d}.
  \end{align*}
  Thus, by integrating this inequality over the interval $[t, 2n^2]$ for $t \in [0, n^2]$, we obtain
  \begin{align*}
    \Norm{v_t^{\pm}}{1, B(n)}
    \;\leq\;
    c\, 
    \int_0^{2n^2}
      \frac{\scpr{|\nabla \eta|}{\om_t |\nabla v_t^{\pm}|}{E}}{|B(2n)|}\;
    \md t
    \,+\, c\, \Norm{v^{\pm}}{1, 1, Q(2n)}.
  \end{align*}
  Further, note that $|\nabla v_t^{\pm}(e)| \leq 2 |\nabla u_t(e)|$ and $|v_t^{\pm}(x)| \leq 2 |u_t(x)| + 1$. Hence,
  \begin{align*}
    &\Norm{n^{-1} \Phi^j}{1, \infty, Q(n)}
    \\[.5ex]
    &\mspace{36mu}\leq\;
    \frac{cn}{2n^2}
    \int_0^{2n^2} \mspace{-18mu}
      \frac{1}{|B(2n)|}\! \sum_{\substack{x \in B(2n) \\ y \sim x}}\mspace{-9mu}
      \om_t(x, y) |u_t(x) - u_t(y)|\;
    \md t
    +
    c \big(1 + \Norm{u}{1,1,Q(2n)}\big)
    \\[.5ex]
    &\mspace{30mu}\overset{\!\!\!\eqref{eq:harm_coord2}\!\!\!}{\;=\;}
    \frac{c}{2n^2}
    \int_0^{2n^2} \mspace{-18mu}
      \frac{1}{|B(2n)|}\! \sum_{\substack{x \in B(2n) \\ y \sim 0}}\mspace{-9mu}
      \om_0(0, y) |\Phi_0^j(\om, y)| \circ \tau_{t,x}\;
    \md t
    +
    c \big(1 + \Norm{n^{-1} \Phi^j}{1,1,Q(2n)}\big).
  \end{align*}
  Since $\Phi_0^{j} \in L_{\mathrm{cov}}^2$, $\mean[\om(e)] < \infty$ and $\Norm{n^{-1} \Phi^j}{1,1,Q(2n)} \leq c(1 + \Norm{n^{-1} \chi^j}{1,1,Q(2n)})$, an application of the ergodic theorem, the Cauchy-Schwarz inequality and Proposition~\ref{prop:sublinearity:l1} completes the proof of \eqref{eq:bound:claim2}.
\end{proof}

\appendix
\section{Technical estimates}
For the reader's convenience we provide some technical estimates needed in Section~\ref{sec:mos_it} in order to process the Moser iteration.  We refer to \cite[Appendix~A]{ADS15} for a proof.  In a sense, they may be regarded as a replacement for a discrete chain rule.
\begin{lemma}
  For $a\in \bbR$, we write $\tilde a^{\al}:=|a|^{\al} \cdot \sign a$ for any $\al \in \bbR \setminus \{0\}$.
  \begin{enumerate}[ (i) ]
    \item For all $a,b \in \bbR$ and any $\al, \be \ne 0$,
      \begin{align}\label{eq:A1:chain:ub1}
        \big| \tilde a^{\al} - \tilde b^{\al} \big|
        \;\leq\;
        \Big(1 \vee \Big|\frac{\al}{\be}\Big| \Big) \,
        \big|\tilde{a}^{\be} - \tilde{b}^{\be}\big|\,
        \big(\,|a|^{\al-\be} + |b|^{\al-\be} \big).
      \end{align}
    \item For all $a,b \in \bbR$ and any $\al > 1/2$,
      \begin{align}\label{eq:A1:pol:ub}
        \big(\tilde{a}^{\al} - \tilde{b}^{\al}\big)^2
        \;\leq\;
        \bigg|\frac{\al^2}{2\al-1}\bigg|\, \big(a - b\big)\,
        \big(\tilde{a}^{2\al-1} - \tilde{b}^{2\al-1}\big).
      \end{align}
      In particular, if $a,b \in \bbR_+$ then \eqref{eq:A1:pol:ub} holds for all $\al \not\in \{0,1/2\}$.
      %
    % \item \textcolor{red}{Not used so far.} For all $a,b \in \bbR$ and any $\al,\be \geq 0$,
    %   %
    %   \begin{align}\label{eq:A1:chain:lo}
    %     \big(|a|^{\al} + |b|^{\al}\big)\,
    %     \big|\tilde{a}^{\be} -\, \tilde{b}^{\be}\big|
    %     \;\leq\;
    %     2\, \big|\tilde{a}^{\al+\be} -\, \tilde{b}^{\al+\be}\big|.
    %   \end{align}
    %   %
    \item For all $a,b \in \bbR$ and any $\al \geq 1/2$,
      \begin{align}\label{eq:A1:chain:ub2}
        \big(|a|^{2\al-1} + |b|^{2\al-1}\big)\, \big|a - b\big|
        \;\leq\;
        4\,
        \big|\tilde a^{\al} - \tilde b^{\al}\big|\,
        \big(|a|^{\al} + |b|^{\al}\big).
      \end{align}
  \end{enumerate}
\end{lemma}

\subsubsection*{Acknowledgment} We thank an anonymous referee for the extremely careful reading and a number of very constructive suggestions to improve an earlier version of the paper.  The proof of Proposition~\ref{prop:sublinearity:l1} resulted from a discussion between the second author and P.\ Bella, B.\ Fehrmann and F.\ Otto, for which we would like to thank them.

\bibliographystyle{abbrv}
\bibliography{literature}

\end{document}